\documentclass[10pt]{amsart}
\usepackage{latexsym}
\usepackage{ulem}
\usepackage[dvips]{color}
\usepackage{float}
\usepackage{color}
\usepackage{afterpage} 
\usepackage [pdftex]{graphicx}
\usepackage{epsfig}
\usepackage{amsmath}
\usepackage{amsthm}
\usepackage{amsfonts}
\usepackage{amssymb}
\usepackage{setspace}
\usepackage{tikz}
\usetikzlibrary{arrows}
\usetikzlibrary{matrix}
\usepackage{subfigure}
\usepackage[margin=1.5in]{geometry}

\newtheorem{definition}{Definition}
\newtheorem{proposition}{Proposition}
\newtheorem{example}{Example}

\usepackage{url}

\begin{document}

\title{Relational PK-Nets for Transformational Music Analysis}

\author{Alexandre Popoff}
\address{119 Rue de Montreuil, 75011 Paris and IRCAM}
\email{al.popoff@free.fr}

\author{Moreno Andreatta}
\address{IRCAM/CNRS/UPMC and IRMA-GREAM, Universit\'e de Strasbourg}
\email{Moreno.Andreatta@ircam.fr, andreatta@math.unistra.fr}

\author{Andr\'ee Ehresmann}
\address{Universit\'e de Picardie, LAMFA}
\email{andree.ehresmann@u-picardie.fr}


\begin{abstract}
In the field of transformational music theory, which emphasizes the possible transformations between musical objects, Klumpenhouwer networks (K-Nets) constitute a useful framework with connections in both group theory and graph theory. Recent attempts at formalizing K-Nets in their most general form have evidenced a deeper connection with category theory. These formalizations use diagrams in sets, i.e. functors $\mathbf{C} \to \mathbf{Sets}$ where $\mathbf{C}$ is often a small category, providing a general framework for the known group or monoid actions on musical objects. However, following the work of Douthett and Cohn, transformational music theory has also relied on the use of relations between sets of the musical elements. Thus, K-Net formalizations have to be further extended to take this aspect into account. This work proposes a new framework called relational PK-Nets, an extension of our previous work on Poly-Klumpenhouwer networks (PK-Nets), in which we consider diagrams in $\mathbf{Rel}$ rather than $\mathbf{Sets}$. We illustrate the potential of relational PK-Nets with selected examples, by analyzing pop music and revisiting the work of Douthett and Cohn.
\end{abstract}

\maketitle

\section{From K-Nets and PK-Nets to relational PK-Nets}

We begin this section by recalling the definition of a Poly-Klumpenhouwer Network (PK-Net), and then discuss about its limitations as a motivation for introducing relational PK-Nets. 

\subsection{The categorical formalization of Poly-Klumpenhouwer Networks (PK-Nets)}

Following the work of Lewin \cite{Lewin1982,Lewin1987}, transformational music theory has progressively shifted the music-theoretical and analytical focus from the ``object-oriented" musical content to the operational musical process. As such, transformations between musical elements are emphasized. In the original framework of Lewin, the set of transformations often form a group, with a corresponding group action on the set of musical objects. Within this framework, Klumpenhouwer Networks (henceforth K-nets) have stressed the deep synergy between set-theoretical and transformational approaches thanks to their anchoring in both group and graph theory, as observed by many scholars \cite{Nolan2007}. We recall that a K-Net is informally defined as a labelled graph, wherein the labels of the vertices belong to the set of pitch classes, whereas arrows are labelled with possible transformations between these pitch classes. An example of a K-Net is given in Figure \ref{fig:KNet}. Klumpenhouwer networks allow one to conveniently visualize at once the musical elements and the possible transformations between them.

Following David Lewin's \cite{Lewin1990} and Henry Klumpenhouwer's \cite{Klumpenhouwer1991} original group-theoretical description, theoretical studies have mostly focused until now on the underlying algebraic methods related to the automorphisms of the T/I group or of the more general T/M affine group \cite{Lewin1990}, \cite{Klumpenhouwer1998}. Following the very first attempt by Mazzola and Andreatta at formalizing K-nets in a more general categorical setting as limits of diagrams within framework of denotators \cite{Mazzola-Andreatta}, we have recently proposed a categorical construction, called Poly-Klumpenhouwer Networks (henceforth PK-nets), which generalizes the notion of K-nets in various ways \cite{PopoffMCM2015,Popoff2016}.

\begin{figure}
\begin{center}
\begin{tikzpicture}[scale=1]
	\node (A) at (0,0) {$C$};
	\node (B) at (1,-1.3) {$E$};
	\node (C) at (0,-2.3) {$G$};
	\draw[->,>=latex] (A) to node[shift={(0.3,0.15)}]{$I_4$} (B) ;
	\draw[->,>=latex] (B) to node[shift={(0.3,-0.15)}]{$T_3$} (C) ;
	\draw[->,>=latex] (A) to node[left,midway]{$I_7$} (C) ;
\end{tikzpicture}
\end{center}
\caption{A Klumpenhouwer network (K-Net) describing a major triad. The arrows are labelled with specific transformations in the $T/I$ group relating the pitch classes.}
\label{fig:KNet}
\end{figure}
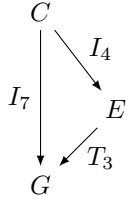

We begin by recalling the definition of a PK-Net, introduced first in \cite{PopoffMCM2015}.

\begin{definition}
Let $\mathbf{C}$ be a category, and $S$ a functor from $\mathbf{C}$ to the category $\mathbf{Sets}$. Let $\Delta$ be a small category and $R$ a functor from $\Delta$ to $\mathbf{Sets}$ with non-void values. A PK-net of form $R$ and of support $S$ is a 4-tuple $(R,S,F,\phi)$, in which 
\begin{itemize}
\setlength\itemsep{0.05em}
\item{
$F$ is a functor from $\Delta$ to $\mathbf{C}$,}
\item{and $\phi$ is a natural transformation from $R$ to $SF$.}
\end{itemize}
\end{definition}

A PK-net $(R,S,F,\phi)$ can be summed up by the following diagram. We detail below the importance of each element of this diagram with respect to transformational music analysis.

\begin{center}
\begin{tikzpicture}
	\node (A) at (0,0) {$\Delta$};
	\node (B) at (3,0) {$\mathbf{C}$};
	\node (C) at (1.5,-2) {$\mathbf{Sets}$};
	\draw[->,>=latex] (A) -- (B) node[above,midway]{$F$};
	\draw[->,>=latex] (B) -- (C) node[right,midway] (D) {$S$} ;
	\draw[->,>=latex] (A) -- (C) node[left,midway]  (E) {$R$};
	
	\draw[->,>=latex,dashed] (E) to[bend left] node[above,midway]{$\phi$} (D) ;
\end{tikzpicture}
\end{center}

The category $\mathbf{C}$ and the functor $S \colon \mathbf{C} \to \mathbf{Sets}$ represent the context of the analysis. Traditional transformational music theory commonly relies on a group acting on a given set of objects: the T/I group acting on the set of the twelve pitch classes, the same T/I group acting simply transitively on the set of the 24 major and minor triads, or the PLR group acting simply transitively on the same set, to name a few examples. From a categorical point of view, the data of a group and its action on a set is equivalent to the data of a functor from a single-object category with invertible morphisms to the category of sets. However, this situation can be further generalized by considering any category $\mathbf{C}$ along with a functor $S \colon \mathbf{C} \to \mathbf{Sets}$. The morphisms of the category $\mathbf{C}$ are therefore the musical transformations of interest. For example, Noll \cite{Noll2005} has studied previously a monoid of eight elements and its action on the set of the twelve pitch classes: this can be considered as a single-object category $\mathbf{C}$ with eight non-invertible morphisms along with its corresponding functor $S \colon \mathbf{C} \to \mathbf{Sets}$, where the image of the only object of $\mathbf{C}$ is the set of the twelve pitch classes. 

The category $\Delta$ serves as the abstract skeleton of the PK-Net: as such, its objects and morphisms are abstract entities, which are labelled by mean of the functor $F$ from $\Delta$ to the category $\mathbf{C}$. By explicitly separating the categories $\Delta$ and $\mathbf{C}$, we allow for the same PK-Net skeleton to be interpreted in different contexts. For example, a given category $\mathbf{C}$ may describe the relationships between pitch classes, while another category $\mathbf{C'}$ may describe the relationships between time-spans \cite{Lewin1987}. Different functors $F \colon \Delta \to \mathbf{C}$ or $F' \colon \Delta \to \mathbf{C'}$ will then label the arrows of $\Delta$ differently with transformations from $\mathbf{C}$ or $\mathbf{C'}$, depending on whether the PK-Net describes pitch-classes or time-spans. Two PK-Nets may actually be related by different kinds of \textit{morphisms of PK-Nets}, which have been described previously \cite{PopoffMCM2015,Popoff2016}.

Note that the objects of $\Delta$ do not represent the actual musical elements of a PK-Net: these are introduced by the functor $R$ from $\Delta$ to $\mathbf{Sets}$. This functor sends each object of $\Delta$ to an actual set, which may contain more than a single element, and whose elements abstractly represent the musical objects of study. However, these elements are not yet labelled. In the same way the morphisms of $\Delta$ represent abstract relationships which are given a concrete meaning by the functor $F$, these elements are labelled by the natural transformation $\phi$. The elements in the image of $S$ represent musical entities on which the category $\mathbf{C}$ acts, and one would therefore need a way to connect the elements in the image of $R$ with those in the image of $S$. However, one cannot simply consider a collection of functions between the images of $R$ and the images of $S$ in order to label the musical objects in the PK-Net. Indeed, one must make sure that two elements in the images of $R$ which are related by a function $R(f)$ (with $f$ being a morphism of $\Delta$) actually correspond to two elements in the images of $S$ related by the function $SF(f)$. The purpose of the natural transformation $\phi$ is thus to ensure the coherence of the whole diagram.

While PK-Nets have been so far defined in $\mathbf{Sets}$, there is \textit{a priori} no restriction on the category that should be used. For example, the approach of Mazzola and Andreatta in \cite{Mazzola-Andreatta} uses modules and categories of presheaves. As noticed in \cite{PopoffMCM2015}, PK-Nets could also be defined in such categories. Alternatively, the purpose of this paper is to consider \textit{relational PK-Nets} in which the category $\mathbf{Sets}$ of sets and functions between them is replaced by the category $\mathbf{Rel}$ of sets and binary relations between them. The reasons for doing so are detailed in the next sections.

\subsection{Limitations of PK-Nets}

The definition of PK-Nets introduced above leads to musical networks of greater generality than traditional Klumpenhouwer networks, as was shown previously in \cite{PopoffMCM2015,Popoff2016}. In particular, it allows one to study networks of sets of different cardinalities, not necessarily limited to singletons. We recall here a prototypical example showing how a dominant seventh chord may be obtained from the transformation of an underlying major chord, with an added seventh.

\begin{example}
Let $\mathbf{C}$ be the $T/I$ group, considered as a single-object category, and consider its natural action on the set $\mathbb{Z}_{12}$ of the twelve pitch classes (with the usual semi-tone encoding), which defines a functor $S \colon T/I \to \mathbf{Sets}$.
Let $\Delta$ define the order of the ordinal number $\mathbf{2}$, i.e. the category with two objects $X$ and $Y$ and precisely one morphism $f \colon X \to Y$, and consider the functor $F \colon \Delta \to T/I$ which sends $f$ to $T_4$. 

Consider now a functor $R \colon \Delta \to \mathbf{Sets}$ such that $R(X)=\{x_1,x_2,x_3\}$ and $R(Y)=\{y_1,y_2,y_3,y_4\}$, and such that $R(f)(x_i)=y_i$, for $1 \leq i \leq 3$. Consider the natural transformation $\phi$ such that $\phi_X(x_1)=C$, $\phi_X(x_2)=E$, $\phi_X(x_3)=G$, and $\phi_Y(y_1)=E$, $\phi_Y(y_2)=G_{\sharp}$, $\phi_Y(y_3)=B$, and $\phi_Y(y_4)=D$. Then $(R,S,F,\phi)$ is a PK-net of form $R$ and support $S$ which describes the transposition of the $C$ major triad to the $E$ major triad subset of the dominant seventh $E^7$ chord. 
\end{example}

However, one specific limitation of this approach appears quickly: whereas transformations between sets of increasing cardinalities can easily be modeled in this framework, transformations between sets of decreasing cardinalities sometimes cannot. Consider for example a PK-Net $(R,S,F,\phi)$ with a category $\Delta$ with at least two objects $X$ and $Y$ and a morphism $f \colon X \to Y$ between them, a functor $R \colon \mathbf{C} \to \mathbf{Sets}$ such that $\left\vert{R(X)}\right\vert>1$ and $\left\vert{R(Y)}\right\vert>1$, and a group $\mathbf{C}$ with a functor $S \colon \mathbf{C} \to \mathbf{Sets}$. Consider two elements $x_1$ and $x_2$ of $R(X)$, an element $y$ of $R(Y)$, and assume that we have $R(f)(x_1)=R(f)(x_2)=y$. By definition of the PK-Net, we must have $SF(f)(\phi_X(x_1))=SF(f)(\phi_X(x_2))=\phi_Y(y)$, and since $\mathbf{C}$ is a group this imposes $\phi_X(x_1)=\phi_X(x_2)$, i.e. the two musical objects $x_1$ and $x_2$ must have the same labels.

As an example, there is no possibility to define a PK-Net showing the inverse transformation from a dominant seventh $E^7$ chord to a $C$ major triad. If one tries to define a new PK-Net $(R',S,F',\phi')$ such that $F'(f)=T_8$, $R'(X)=\{x_1,x_2,x_3,x_4\}$, $R'(Y) = \{y_1,y_2,y_3\}$, and with a natural transformation $\phi'$ such that $\phi'_X(x_1)=E$, $\phi'_X(x_2)=G_{\sharp}$, $\phi'_X(x_3)=B$, and $\phi'_Y(y_1)=C$, $\phi_Y(y_2)=E$, $\phi_Y(y_3)=G$, then one quickly sees that no function $R'(f)$ can exist which would satisfy the requirement that $\phi'$ is a natural transformation from $R$ to $SF'$, since all the elements constituting the seventh chord have different labels in $\mathbb{Z}_{12}$.

In a possible way to resolve this problem, and from the point of view that the $E^7$ dominant seventh chord consists of an $E$ major chord with an added $D$ note, we would intuitively like to ``forget'' about the $D$ and consider only the transformation of $x_1$, $x_2$, and $x_3$ in $R(X)$ to $y_1$, $y_2$, and $y_3$ respectively in $R'(Y)$. In other words, we would like to consider \textit{partial functions} between sets, instead of ordinary ones. In order to do so, one must abandon the category $\mathbf{Sets}$ and choose a category which makes it possible to consider such morphisms. This simple example motivates the introduction in this paper of \textit{relational PK-Nets} based on the category $\mathbf{Rel}$ of sets and binary relations between them. Although there also exists a category $\mathbf{Par}$ whose objects are sets and morphisms are partial functions between them, the use of $\mathbf{Rel}$ includes the case of partial functions as well as even more general applications, as will be seen in the next sections.

\subsection{The use of relations in transformational music theory}

The use of relations between musical objects figures prominently in the recent literature on music theory. Perhaps one of the most compelling examples is the work of Douthett and Cohn on parsimonious voice-leading and its subsequent formalization in the form of parsimonious graphs \cite{Douthett1998,Cohn2012}. In order to formalize the notion of parsimony, Douthett introduced in \cite{Douthett1998} the $\mathcal{P}_{m,n}$ relation between two pc-sets, the definition of which we recall here. We recall that a \textit{pitch class set (pc-set)} is a set in $\mathbb{Z}_{12}$ (which encodes the different pitch classes with the usual semitone encoding).

\begin{definition}
Let $O=\{x_k, 0 \leq k \leq l\}$ and $O'=\{y_k, 0 \leq k \leq l\}$ be two pc-sets of equal cardinality $l$. We say that $O$ and $O'$ are $\mathcal{P}_{m,n}$-related if there exists a set $Z = \{z_k, 0 \leq k \leq m+n-1\}$ and a bijection $\tau \colon O \to O'$ such that we have $O-O'=Z$, $\tau(x_i)=y_i$ if $x_i \in O \cap O'$, and 
\begin{itemize}
\item{$\tau(x_i)=z_i \pm 1$ if $0 \leq i \leq m-1$, and}
\item{$\tau(x_i)=z_i \pm 2$ if $m \leq i \leq m+n-1$.} 
\end{itemize}
\end{definition}

In other words, if $O'$ is $\mathcal{P}_{m,n}$-related to $O$, $m$ pitch-classes in $O$ move by a semi-tone, while $n$ pitch-classes move by a whole tone, the rest of the pitch classes being identical. Note that $\mathcal{P}_{m,n}$ is a symmetric relation. From this definition, Douthett defines a \textit{parsimonious graph for a $\mathcal{P}_{m,n}$ relation} on a set $H$ of pc-sets as the graph whose set of vertices is $H$ and whose set of edges is the set $\{(O,O') \mid O \in H, O' \in H, O \mathcal{P}_{m,n} O'\}$.
A notable example is the ``Cube Dance'' graph, which is the parsimonious graph for the $\mathcal{P}_{1,0}$ relation (i.e. the voice-leading relation between chords resulting from the ascending or descending movement of a single pitch class by a semitone) on the set $H$ of 28 elements containing the 24 major and minor triads as well as the four augmented triads. This graph is represented on Figure \ref{fig:CubeDance}. One should note that this graph contains the subgraphs defined on the set $H$ by the neo-Riemannian operations $L$ and $P$ viewed as relations. Indeed, given $O$ and $O'$ in $H$ such that $O'=L(O)$ or $O'=P(O)$, one can immediately verify by definition of these neo-Riemannian operations that we have $O \mathcal{P}_{1,0} O'$. This subgraph is called ``HexaCycles'' by Douthett. The Cube Dance adds to ``HexaCycles'' the possible relations between the augmented triads and the hexatonic cycles, which, as commented by Douthett, ``\textit{serve as the couplings between hexatonic cycles and function nicely as a way of modulating between hexatonic sets}''.

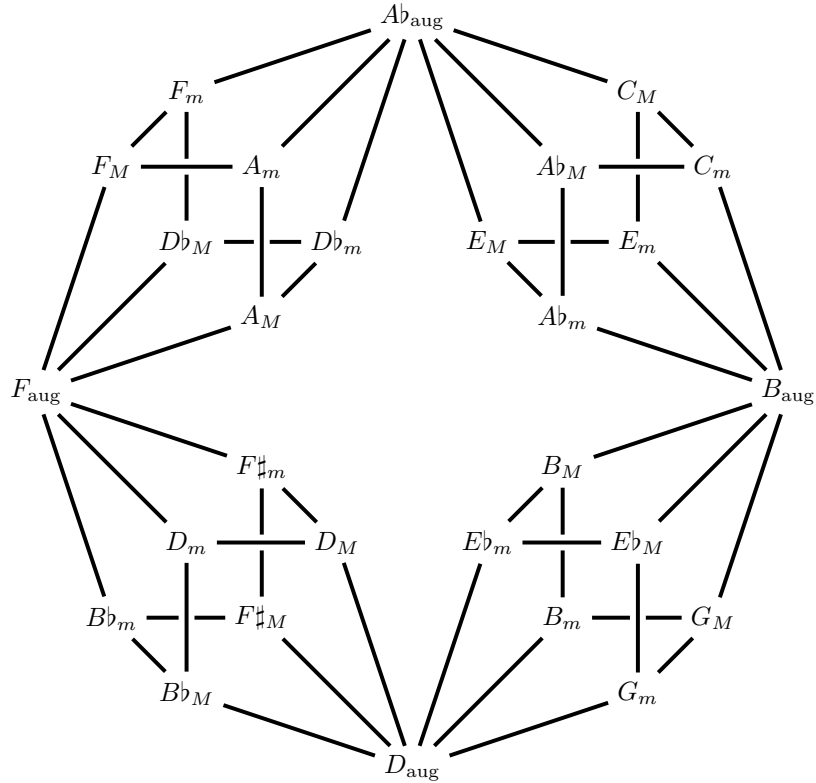
\begin{figure}
\begin{center}
\begin{tikzpicture}[scale=1.0]

	\node (CM) at (3,4) {$C_M$};
	\node (Cm) at (4,3) {$C_m$};
	\node (Em) at (3,2) {$E_m$};
	\node (EM) at (1,2) {$E_M$};
	\node (Abm) at (2,1) {$A\flat_m$};
	\node (AbM) at (2,3) {$A\flat_M$};
	
	\draw[-,>=latex, line width=1.5] (EM) to (Em) ;
	\draw[-,>=latex, line width=1.5] (Em) to (CM) ;
	\draw[-,>=latex, line width=1.5] (CM) to (Cm) ;
		\draw[-,>=latex, line width=6, color=white] (Cm) to (AbM) ;
	\draw[-,>=latex, line width=1.5] (Cm) to (AbM) ;
		\draw[-,>=latex, line width=6, white] (AbM) to (Abm) ;
	\draw[-,>=latex, line width=1.5] (AbM) to (Abm) ;
	\draw[-,>=latex, line width=1.5] (Abm) to (EM) ;

	\node (Fm) at (-3,4) {$F_m$};
	\node (FM) at (-4,3) {$F_M$};
	\node (DbM) at (-3,2) {$D\flat_M$};
	\node (Dbm) at (-1,2) {$D\flat_m$};
	\node (AM) at (-2,1) {$A_M$};
	\node (Am) at (-2,3) {$A_m$};
	\draw[-,>=latex, line width=1.5] (Dbm) to (DbM) ;
	\draw[-,>=latex, line width=1.5] (DbM) to (Fm) ;
	\draw[-,>=latex, line width=1.5] (Fm) to (FM) ;
		\draw[-,>=latex, line width=6, white] (FM) to (Am) ;
	\draw[-,>=latex, line width=1.5] (FM) to (Am) ;
		\draw[-,>=latex, line width=6, white] (Am) to (AM) ;
	\draw[-,>=latex, line width=1.5] (Am) to (AM) ;
	\draw[-,>=latex, line width=1.5] (AM) to (Dbm) ;

	\node (BbM) at (-3,-4) {$B\flat_M$};
	\node (Bbm) at (-4,-3) {$B\flat_m$};
	\node (Dm) at (-3,-2) {$D_m$};
	\node (DM) at (-1,-2) {$D_M$};
	\node (Fsm) at (-2,-1) {$F\sharp_m$};
	\node (FsM) at (-2,-3) {$F\sharp_M$};
	\draw[-,>=latex, line width=1.5] (Fsm) to (FsM) ;
	\draw[-,>=latex, line width=1.5] (FsM) to (Bbm) ;
	\draw[-,>=latex, line width=1.5] (Bbm) to (BbM) ;
		\draw[-,>=latex, line width=6, white] (BbM) to (Dm) ;
	\draw[-,>=latex, line width=1.5] (BbM) to (Dm) ;
		\draw[-,>=latex, line width=6, white] (Dm) to (DM) ;
	\draw[-,>=latex, line width=1.5] (Dm) to (DM) ;
	\draw[-,>=latex, line width=1.5] (DM) to (Fsm) ;

	\node (Gm) at (3,-4) {$G_m$};
	\node (GM) at (4,-3) {$G_M$};
	\node (EbM) at (3,-2) {$E\flat_M$};
	\node (Ebm) at (1,-2) {$E\flat_m$};
	\node (BM) at (2,-1) {$B_M$};
	\node (Bm) at (2,-3) {$B_m$};
	\draw[-,>=latex, line width=1.5] (GM) to (Bm) ;
	\draw[-,>=latex, line width=1.5] (Bm) to (BM) ;
	\draw[-,>=latex, line width=1.5] (BM) to (Ebm) ;
		\draw[-,>=latex, line width=6, white] (Ebm) to (EbM) ;
	\draw[-,>=latex, line width=1.5] (Ebm) to (EbM) ;
		\draw[-,>=latex, line width=6, white] (EbM) to (Gm) ;
	\draw[-,>=latex, line width=1.5] (EbM) to (Gm) ;
	\draw[-,>=latex, line width=1.5] (Gm) to (GM) ;

	\node (Faug) at (-5,0) {$F_{\text{aug}}$};
	\node (Baug) at (5,0) {$B_{\text{aug}}$};
	\node (Abaug) at (0,5) {$A\flat_{\text{aug}}$};
	\node (Daug) at (0,-5) {$D_{\text{aug}}$};
	
	\draw[-,>=latex, line width=1.5] (Fsm) to (Faug) ;
	\draw[-,>=latex, line width=1.5] (Dm) to (Faug) ;
	\draw[-,>=latex, line width=1.5] (Bbm) to (Faug) ;
	\draw[-,>=latex, line width=1.5] (FM) to (Faug) ;
	\draw[-,>=latex, line width=1.5] (DbM) to (Faug) ;
	\draw[-,>=latex, line width=1.5] (AM) to (Faug) ;
	
	\draw[-,>=latex, line width=1.5] (Fm) to (Abaug) ;
	\draw[-,>=latex, line width=1.5] (Am) to (Abaug) ;
	\draw[-,>=latex, line width=1.5] (Dbm) to (Abaug) ;
	\draw[-,>=latex, line width=1.5] (CM) to (Abaug) ;
	\draw[-,>=latex, line width=1.5] (AbM) to (Abaug) ;
	\draw[-,>=latex, line width=1.5] (EM) to (Abaug) ;

	\draw[-,>=latex, line width=1.5] (BM) to (Baug) ;
	\draw[-,>=latex, line width=1.5] (EbM) to (Baug) ;
	\draw[-,>=latex, line width=1.5] (GM) to (Baug) ;
	\draw[-,>=latex, line width=1.5] (Cm) to (Baug) ;
	\draw[-,>=latex, line width=1.5] (Em) to (Baug) ;
	\draw[-,>=latex, line width=1.5] (Abm) to (Baug) ;

	\draw[-,>=latex, line width=1.5] (DM) to (Daug) ;
	\draw[-,>=latex, line width=1.5] (FsM) to (Daug) ;
	\draw[-,>=latex, line width=1.5] (BbM) to (Daug) ;
	\draw[-,>=latex, line width=1.5] (Ebm) to (Daug) ;
	\draw[-,>=latex, line width=1.5] (Bm) to (Daug) ;
	\draw[-,>=latex, line width=1.5] (Gm) to (Daug) ;
	
\end{tikzpicture}
\end{center}
\caption{Douthett's ``Cube Dance'' graph as the parsimonious graph for the $\mathcal{P}_{1,0}$ relation on the set of the 24 major and minor triads and the four augmented triads.}
\label{fig:CubeDance}
\end{figure}

Douthett also introduced the ``Weitzmann's Waltz'' graph, which corresponds to the parsimonious graph for the $\mathcal{P}_{2,0}$ relation (i.e. the voice-leading relation between chords resulting from the ascending or descending movement of two pitch classes by a semitone) on the set $H$ of 28 elements containing the 24 major and minor triads as well as the four augmented triads. This graph is represented on Figure \ref{fig:WeitzmannWaltz}.

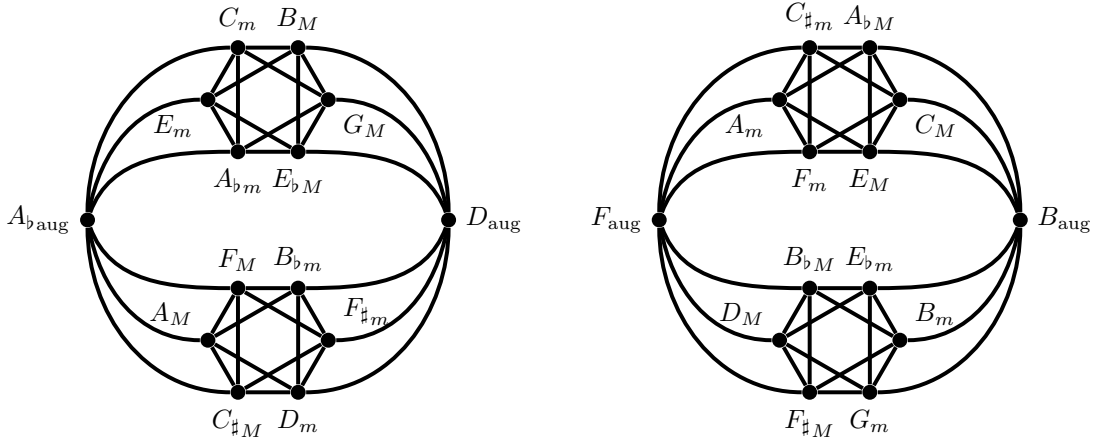
\begin{figure}
\begin{center}
\begin{tikzpicture}[scale=0.8]
	\node (GM) at (-5+1,0) [circle,fill=black,inner sep=2pt,label=below right:$G_M$]{};
	\node (BM) at (-5+0.5,0.866) [circle,fill=black,inner sep=2pt,label=above:$B_M$]{};
	\node (Cm) at (-5-0.5,0.866) [circle,fill=black,inner sep=2pt,label=above:$C_m$]{};
	\node (Em) at (-5-1,0) [circle,fill=black,inner sep=2pt,label=below left:$E_m$]{};
	\node (Abm) at (-5-0.5,-0.866) [circle,fill=black,inner sep=2pt,label=below:${A_\flat}_m$]{};
	\node (EbM) at (-5+0.5,-0.866) [circle,fill=black,inner sep=2pt,label=below:${E_\flat}_M$]{};

	\node (Fsm) at (-5+1,0-4) [circle,fill=black,inner sep=2pt,label=above right:${F_\sharp}_m$]{};
	\node (Bbm) at (-5+0.5,0.866-4) [circle,fill=black,inner sep=2pt,label=above:${B_\flat}_m$]{};
	\node (FM) at (-5-0.5,0.866-4) [circle,fill=black,inner sep=2pt,label=above:$F_M$]{};
	\node (AM) at (-5-1,0-4) [circle,fill=black,inner sep=2pt,label=above left:$A_M$]{};
	\node (CsM) at (-5-0.5,-0.866-4) [circle,fill=black,inner sep=2pt,label=below:${C_\sharp}_M$]{};
	\node (Dm) at (-5+0.5,-0.866-4) [circle,fill=black,inner sep=2pt,label=below:$D_m$]{};
	
	\draw[-,>=latex, line width=1.5] (GM) to (BM) ;
	\draw[-,>=latex, line width=1.5] (GM) to (EbM) ;
	\draw[-,>=latex, line width=1.5] (GM) to (Cm) ;
	\draw[-,>=latex, line width=1.5] (GM) to (Abm) ;
	\draw[-,>=latex, line width=1.5] (BM) to (Cm) ;
	\draw[-,>=latex, line width=1.5] (EbM) to (Abm) ;
	\draw[-,>=latex, line width=1.5] (Em) to (Cm) ;
	\draw[-,>=latex, line width=1.5] (Em) to (Abm) ;
	\draw[-,>=latex, line width=1.5] (Em) to (EbM) ;
	\draw[-,>=latex, line width=1.5] (Em) to (BM) ;
	\draw[-,>=latex, line width=1.5] (BM) to (EbM) ;
	\draw[-,>=latex, line width=1.5] (Cm) to (Abm) ;

	\draw[-,>=latex, line width=1.5] (FM) to (Bbm) ;
	\draw[-,>=latex, line width=1.5] (FM) to (AM) ;
	\draw[-,>=latex, line width=1.5] (AM) to (CsM) ;
	\draw[-,>=latex, line width=1.5] (CsM) to (Dm) ;
	\draw[-,>=latex, line width=1.5] (Dm) to (Fsm) ;
	\draw[-,>=latex, line width=1.5] (Fsm) to (Bbm) ;
	\draw[-,>=latex, line width=1.5] (Bbm) to (Dm) ;
	\draw[-,>=latex, line width=1.5] (FM) to (CsM) ;
	\draw[-,>=latex, line width=1.5] (Bbm) to (AM) ;
	\draw[-,>=latex, line width=1.5] (AM) to (Dm) ;
	\draw[-,>=latex, line width=1.5] (Fsm) to (CsM) ;
	\draw[-,>=latex, line width=1.5] (Fsm) to (FM) ;

	\node (Daug) at (-5+3,-2) [circle,fill=black,inner sep=2pt,label=right:$D_{\text{aug}}$]{};
	\node (Abaug) at (-5-3,-2) [circle,fill=black,inner sep=2pt,label=left:${A_\flat}_{\text{aug}}$]{};
	
	\draw[-,>=latex, line width=1.5] (GM) to[out=0, in=100] (Daug) ;
	\draw[-,>=latex, line width=1.5] (BM) to[out=0, in=90] (Daug) ;
	\draw[-,>=latex, line width=1.5] (EbM) to[out=0, in=110] (Daug) ;
	\draw[-,>=latex, line width=1.5] (Cm) to[out=180, in=90] (Abaug) ;
	\draw[-,>=latex, line width=1.5] (Em) to[out=180, in=80] (Abaug) ;
	\draw[-,>=latex, line width=1.5] (Abm) to[out=180, in=70] (Abaug) ;

	\draw[-,>=latex, line width=1.5] (Bbm) to[out=0, in=250] (Daug) ;
	\draw[-,>=latex, line width=1.5] (Fsm) to[out=0, in=260] (Daug) ;
	\draw[-,>=latex, line width=1.5] (Dm) to[out=0, in=270] (Daug) ;
	\draw[-,>=latex, line width=1.5] (FM) to[out=180, in=290] (Abaug) ;
	\draw[-,>=latex, line width=1.5] (AM) to[out=180, in=280] (Abaug) ;
	\draw[-,>=latex, line width=1.5] (CsM) to[out=180, in=270] (Abaug) ;


	\node (CM) at (4.5+1,0) [circle,fill=black,inner sep=2pt,label=below right:$C_M$]{};
	\node (AbM) at (4.5+0.5,0.866) [circle,fill=black,inner sep=2pt,label=above:${A_\flat}_M$]{};
	\node (Csm) at (4.5-0.5,0.866) [circle,fill=black,inner sep=2pt,label=above:${C_\sharp}_m$]{};
	\node (Am) at (4.5-1,0) [circle,fill=black,inner sep=2pt,label=below left:$A_m$]{};
	\node (Fm) at (4.5-0.5,-0.866) [circle,fill=black,inner sep=2pt,label=below:$F_m$]{};
	\node (EM) at (4.5+0.5,-0.866) [circle,fill=black,inner sep=2pt,label=below:$E_M$]{};

	\node (Bm) at (4.5+1,0-4) [circle,fill=black,inner sep=2pt,label=above right:$B_m$]{};
	\node (Ebm) at (4.5+0.5,0.866-4) [circle,fill=black,inner sep=2pt,label=above:${E_\flat}_m$]{};
	\node (BbM) at (4.5-0.5,0.866-4) [circle,fill=black,inner sep=2pt,label=above:${B_\flat}_M$]{};
	\node (DM) at (4.5-1,0-4) [circle,fill=black,inner sep=2pt,label=above left:$D_M$]{};
	\node (FsM) at (4.5-0.5,-0.866-4) [circle,fill=black,inner sep=2pt,label=below:${F_\sharp}_M$]{};
	\node (Gm) at (4.5+0.5,-0.866-4) [circle,fill=black,inner sep=2pt,label=below:$G_m$]{};
	
	\draw[-,>=latex, line width=1.5] (CM) to (AbM) ;
	\draw[-,>=latex, line width=1.5] (AbM) to (Csm) ;
	\draw[-,>=latex, line width=1.5] (Csm) to (Am) ;
	\draw[-,>=latex, line width=1.5] (Am) to (Fm) ;
	\draw[-,>=latex, line width=1.5] (Fm) to (EM) ;
	\draw[-,>=latex, line width=1.5] (EM) to (CM) ;
	\draw[-,>=latex, line width=1.5] (CM) to (Csm) ;
	\draw[-,>=latex, line width=1.5] (Csm) to (Fm) ;
	\draw[-,>=latex, line width=1.5] (Fm) to (CM) ;
	\draw[-,>=latex, line width=1.5] (AbM) to (Am) ;
	\draw[-,>=latex, line width=1.5] (Am) to (EM) ;
	\draw[-,>=latex, line width=1.5] (EM) to (AbM) ;

	\draw[-,>=latex, line width=1.5] (Bm) to (Ebm) ;
	\draw[-,>=latex, line width=1.5] (Ebm) to (BbM) ;
	\draw[-,>=latex, line width=1.5] (BbM) to (DM) ;
	\draw[-,>=latex, line width=1.5] (DM) to (FsM) ;
	\draw[-,>=latex, line width=1.5] (FsM) to (Gm) ;
	\draw[-,>=latex, line width=1.5] (Gm) to (Bm) ;
	\draw[-,>=latex, line width=1.5] (Bm) to (BbM) ;
	\draw[-,>=latex, line width=1.5] (BbM) to (FsM) ;
	\draw[-,>=latex, line width=1.5] (FsM) to (Bm) ;
	\draw[-,>=latex, line width=1.5] (Ebm) to (DM) ;
	\draw[-,>=latex, line width=1.5] (DM) to (Gm) ;
	\draw[-,>=latex, line width=1.5] (Gm) to (Ebm) ;

	\node (Baug) at (4.5+3,-2) [circle,fill=black,inner sep=2pt,label=right:$B_{\text{aug}}$]{};
	\node (Faug) at (4.5-3,-2) [circle,fill=black,inner sep=2pt,label=left:$F_{\text{aug}}$]{};
	
	\draw[-,>=latex, line width=1.5] (AbM) to[out=0, in=90] (Baug) ;
	\draw[-,>=latex, line width=1.5] (CM) to[out=0, in=100] (Baug) ;
	\draw[-,>=latex, line width=1.5] (EM) to[out=0, in=110] (Baug) ;
	\draw[-,>=latex, line width=1.5] (Csm) to[out=180, in=90] (Faug) ;
	\draw[-,>=latex, line width=1.5] (Am) to[out=180, in=80] (Faug) ;
	\draw[-,>=latex, line width=1.5] (Fm) to[out=180, in=70] (Faug) ;

	\draw[-,>=latex, line width=1.5] (Ebm) to[out=0, in=250] (Baug) ;
	\draw[-,>=latex, line width=1.5] (Bm) to[out=0, in=260] (Baug) ;
	\draw[-,>=latex, line width=1.5] (Gm) to[out=0, in=270] (Baug) ;
	\draw[-,>=latex, line width=1.5] (BbM) to[out=180, in=290] (Faug) ;
	\draw[-,>=latex, line width=1.5] (DM) to[out=180, in=280] (Faug) ;
	\draw[-,>=latex, line width=1.5] (FsM) to[out=180, in=270] (Faug) ;

\end{tikzpicture}
\end{center}
\caption{Douthett's ``Weitzmann's Waltz'' graph as the parsimonious graph for the $\mathcal{P}_{2,0}$ relation on the set of the 24 major and minor triads and the four augmented triads.}
\label{fig:WeitzmannWaltz}
\end{figure}


Whereas relations can be described as graphs, which can then be used for musical applications, \textit{transformational} analysis using relations is however trickier to define than in the case of groups and group actions. In the framework of Lewin, the image of a given $x$ by the group action of an element $g$ of a group $G$ is determined unambiguously. Hence, one can speak about ``applying the operation $g$ to the musical element $x$'', or speaking about ``the image of the musical element $x$ by the operation $g$''. Assume instead that, given a relation $\mathcal{R}$ between two sets $X$ and $Y$ and an element $x$ of $X$, there exists multiple $y$ in $Y$ such that we have $x \mathcal{R} y$. How can then one define ``the image of the musical element $x$ under the relation $\mathcal{R}$'' ? This question motivates the use of the more general framework of \textit{relational PK-Nets}, which we define in the next section.

\section{Defining Relational PK-Nets}

Before introducing the definition of relational PK-Nets, we recall basic facts about relations and the associated category $\mathbf{Rel}$.

\subsection{The 2-category $\mathbf{Rel}$}

We first recall some basic definitions about relations.

\begin{definition}
Let $X$ and $Y$ be two sets. A binary relation $\mathcal{R}$ between $X$ and $Y$ is a subset of the cartesian product $X \times Y$. We say that $y \in Y$ is related to $x \in X$ by $\mathcal{R}$, which is notated as $x \mathcal{R} y$, if $(x,y) \in \mathcal{R}$.
\end{definition}

\begin{definition}
Let $\mathcal{R}$ be a relation between two sets $X$ and $Y$. We say that $\mathcal{R}$ is left-total if, for all $x$ in $X$, there exists at least one $y$ in $Y$ such that we have $x \mathcal{R} y$.
\end{definition}

\begin{definition}
Let $X$ and $Y$ be two sets, and $\mathcal{R}$ and $\mathcal{R'}$ be two relations between them. The relation $\mathcal{R}$ is said to be included in $\mathcal{R'}$ if $x \mathcal{R} y$ implies $x \mathcal{R'} y$, for all pairs $(x,y) \in X \times Y$.
\end{definition}

Relations can be composed according to the definition below.

\begin{definition}
Let $X$, $Y$, and $Z$ be sets, $\mathcal{R}$ be a relation between $X$ and $Y$, and $\mathcal{R'}$ be a relation between $Y$ and $Z$. The composition of $\mathcal{R'}$ and $\mathcal{R}$ is the relation $\mathcal{R''} = \mathcal{R'} \circ \mathcal{R}$ defined by the pairs $(x,z)$ with $x \in X$ and $z \in Z$ such that there exists at least one $y \in Y$ such that $x \mathcal{R} y$, and $y \mathcal{R'} z$.
\end{definition}

By the definition below, sets and binary relations between them form a 2-category.

\begin{definition}
The 2-category $\mathbf{Rel}$ is the category which has
\begin{itemize}
\setlength\itemsep{0.05em}
\item{sets as objects,}
\item{relations as 1-morphisms between them, and}
\item{inclusion as 2-morphisms between relations.}
\end{itemize}
\end{definition}

Notice that the definition of relations includes the particular case of functions between sets: if $\mathcal{R}$ is a relation between two sets $X$ and $Y$, then $\mathcal{R}$ is a function if, given any element $x$ in $X$, there exists exactly one $y$ in $Y$ such that we have $x \mathcal{R} y$. As a consequence, the category $\mathbf{Sets}$ of sets and functions between them is a subcategory of $\mathbf{Rel}$. Notice also that the definition of relations also includes the case of partial functions between sets: if $\mathcal{R}$ is a relation between two sets $X$ and $Y$, it may be possible that, given an element $x$ in $X$, there exists no $y$ in $Y$ such that we have $x \mathcal{R} y$. Hence, the category $\mathbf{Par}$ of sets and partial functions between them is also a subcategory of $\mathbf{Rel}$.

Since $\mathbf{Rel}$ is a 2-category, the usual notion of a functor to $\mathbf{Rel}$ may be replaced by the notion of a \textit{lax functor} to $\mathbf{Rel}$ in order to take into account the 2-morphisms between relations. This notion is defined more precisely below.

\begin{definition}
Let $\mathcal{C}$ be a 1-category. A lax functor $F$ from $\mathcal{C}$ to $\mathbf{Rel}$ is the data of a map
\begin{itemize}
\setlength\itemsep{0.05em}
\item{which sends each object $X$ of $\mathbf{C}$ to an object $F(X)$ of $\mathbf{Rel}$, and}
\item{which sends each morphism $f \colon X \to Y$ of $\mathbf{C}$ to a relation $F(f) \colon F(X) \to F(Y)$ of $\mathbf{Rel}$, such that for each pair $(f,g)$ of composable morphisms $f \colon X \to Y$ and $g \colon Y \to Z$ the image relation $F(g)F(f)$ is included in $F(gf)$.}
\end{itemize}
\end{definition}

A lax functor will be called a 1-functor (coinciding with the usual notion of functor between 1-categories) when $F(g)F(f)=F(gf)$.

Given two lax functors $F$ and $G$ to $\mathbf{Rel}$, the usual notion of a natural transformation $\eta$ between $F$ and $G$ has to be replaced by that of a \textit{lax natural transformation}, which we define below.

\begin{definition}
Let $\mathbf{C}$ be a 1-category, and let $F$ and $G$ be two lax functors from $\mathbf{C}$ to $\mathbf{Rel}$. A lax natural transformation $\eta$ between $F$ and $G$ is the data of a collection of relations $\{\eta_X \colon F(X) \to G(X)\}$ for all objects $X$ of $\mathbf{C}$, such that, for any morphism $f \colon X \to Y$, the relation $\eta_Y F(f)$ is included in the relation $G(f) \eta_X$.
\end{definition}

Finally, there exists a notion of inclusion of lax natural transformations, which we define precisely below.

\begin{definition}
Let $\mathbf{C}$ be a 1-category, let $F$ and $G$ be two lax functors from $\mathbf{C}$ to $\mathbf{Rel}$, and let $\eta$ and $\eta'$ be two lax natural transformation between $F$ and $G$. We say that $\eta$ is included in $\eta'$ if, for any object $X$ of $\mathbf{C}$, the component $\eta_X$ is included in the component $\eta'_X$.
\end{definition}

\subsection{Relational PK-Nets}

With the previous definitions in mind, we now give the formal definition of a relational PK-Net.

\begin{definition}
Let $\mathbf{C}$ be a small 1-category, and $S$ a lax functor from $\mathbf{C}$ to the category $\mathbf{Rel}$. Let $\Delta$ be a small 1-category and $R$ a lax functor from $\Delta$ to $\mathbf{Rel}$ with non-void values. A relational PK-net of form $R$ and of support $S$ is a 4-tuple $(R,S,F,\phi)$, in which 
\begin{itemize}
\setlength\itemsep{0.05em}
\item{$F$ is a functor from $\Delta$ to $\mathbf{C}$,}
\item{and $\phi$ is a lax natural transformation from $R$ to $SF$, such that, for any object $X$ of $\Delta$, the component $\phi_X$ is left-total.}
\end{itemize}
\end{definition}

The definition given above is almost similar to that of a PK-Net in $\mathbf{Sets}$, but the specificities of the 2-category $\mathbf{Rel}$ impose slight adjustments. The first one is the requirement that the functor $R$ shall be a lax functor to $\mathbf{Rel}$ instead of a 1-functor. To see why this is the case, let $X$, $Y$, and $Z$ be objects of $\Delta$, and let $f \colon X \to Y$, $g \colon Y \to Z$, and $h \colon X \to Z$ be morphisms between them, with $h=gf$. Relations are more general than functions: given the relation $R(f)$ between the sets $R(X)$ and $R(Y)$, it is possible that, for a given element $x$ in $R(X)$ there exists multiple elements $y$ in $R(Y)$ such that we have $x R(f) y$, or even none at all. To be consistent, we require that, given $x$ in $R(X)$ and $z$ in $R(Z)$, we have that $x R(g) \circ R(f) z$ implies $x R(h) z$. However, we do not require the strict equality of relations, as it gives more flexibility over the possible relations with the elements of $R(Y)$. The first example below will clarify this notion in the case of the analysis of sets of varying cardinalities, in particular for triads and seventh chords. Note that the same logic requires that $S$ shall be a lax functor to $\mathbf{Rel}$ as well.

The second adjustment corresponds to the requirement that $\phi$ shall be a lax natural transformation from $R$ to $SF$ instead of an ordinary one. Let us recall the role of the functor $S$ and the natural tranformation $\phi$ in the case of PK-Nets in $\mathbf{Sets}$. The lax functor $S$ defines the context of the analysis: for any objects $e$ and $e'$ of $\mathbf{C}$ and a morphism $f \colon e \to e'$ between them, the sets $S(e)$ and $S(e')$ represent all the possible musical entities of interest and the function $S(f)$ represents a transformation between them. Given a set of unnamed musical objects $R(X)$ (with $X$ being an object of the category $\Delta$), the component $\phi_X$ of the natural transformation $\phi$ is a function which ``names'' these objects by their images in $SF(X)$. In the case of relational PK-Nets, since $S(f)$ is a relation instead of a function, it is possible that an element of $S(e)$ may be related to more than one element of $S(e')$, or even none. However, the relations between the actual musical objects under study in the images of the lax functor $R$ may not cover the range of possibilities offered by the relations in the image of $SF$. Thus, we use a lax natural transformation $\phi$ instead of an ordinary one, such that the images by $\phi$ of the relations given through $R$ be included in those given through $SF$. Informally speaking, the lax natural transformation $\phi$ ``selects'' a restricted range of related elements among the possibilities given by the functor $SF$. In addition, we require that each component of $\phi$ be left-total, so that all musical objects are ``named'' in the images of the functor $SF$. All the constitutive elements of a relational PK-Net $(R,S,F\phi)$ are summed-up in the diagram of Figure \ref{fig:RelPKNet}. In addition, different examples in Section 3 of this paper will clarify this notion.

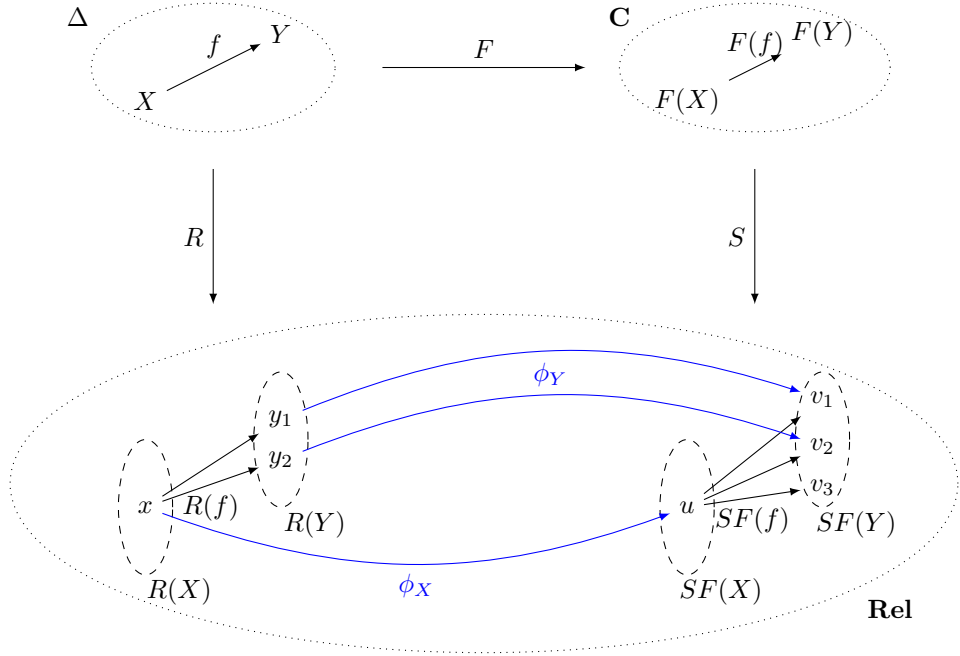
\begin{figure}
\begin{center}
\begin{tikzpicture}[scale=0.9]
	\node (X) at (0,0) {$X$};
	\node (Y) at (2,1) {$Y$};
	\draw [->,>=latex] (X) -- (Y) node[above,midway] {$f$};

	\draw [->,>=latex] (1,-1) -- (1,-3) node[left,midway] {$R$};

	\draw[dashed,rotate=0] (0,-6) ellipse (0.4 and 1);
	\draw[dashed,rotate=0] (2,-5) ellipse (0.4 and 1);
	\node (RX) at (0.5,-7.25) {$R(X)$};
	\node (RY) at (2.5,-6.25) {$R(Y)$};
	\node (x) at (0,-6) {$x$};
	\node (y1) at (2,-4.7) {$y_1$};
	\node (y2) at (2,-5.3) {$y_2$};
	\draw [->,>=latex] (x) -- (y1) node[above,midway] {};
	\draw [->,>=latex] (x) -- (y2) node[below,midway] {$R(f)$};

	\draw [->,>=latex] (3.5,0.5) -- (6.5,0.5) node[above,midway] {$F$};

	\node (FX) at (8,0) {$F(X)$};
	\node (FY) at (10,1) {$F(Y)$};
	\draw [->,>=latex] (FX) -- (FY) node[above,midway] {$F(f)$};

	\draw [->,>=latex] (9,-1) -- (9,-3) node[left,midway] {$S$};

	\draw[dashed,rotate=0] (8,-6) ellipse (0.4 and 1);
	\draw[dashed,rotate=0] (10,-5) ellipse (0.4 and 1);
	\node (SFX) at (8.5,-7.25) {$SF(X)$};
	\node (SFY) at (10.5,-6.25) {$SF(Y)$};
	\node (u) at (8,-6) {$u$};
	\node (v1) at (10,-4.4) {$v_1$};
	\node (v2) at (10,-5.1) {$v_2$};
	\node (v3) at (10,-5.7) {$v_3$};
	\draw [->,>=latex] (u) -- (v1) node[above,midway] {};
	\draw [->,>=latex] (u) -- (v2) node[below,midway] {};
	\draw [->,>=latex] (u) -- (v3) node[below,midway] {$SF(f)$};

	\draw [->,>=latex, color=blue] (x) to[bend right=20] node[below,midway] {$\phi_X$} (u) ;
	\draw [->,>=latex, color=blue] (y1) to[bend left=20] (v1) ;
	\draw [->,>=latex, color=blue] (y2) to[bend left=20] node[above,midway] {$\phi_Y$}(v2) ;

	\draw[dotted,rotate=0] (5,-5.65) ellipse (7 and 2.5);
	\node (Rel) at (11,-7.5) {$\mathbf{Rel}$};
	\draw[dotted,rotate=0] (1,0.5) ellipse (1.8 and 0.95);
	\node (Delta) at (-1,1.25) {$\Delta$};
	\draw[dotted,rotate=0] (9,0.5) ellipse (2 and 0.95);
	\node (C) at (7,1.25) {$\mathbf{C}$};
\end{tikzpicture}
\end{center}
\caption{Diagram showing the constitutive elements of a simple relational PK-Net $(R,S,F,\phi)$.}
\label{fig:RelPKNet}
\end{figure}

We now give a simple example to illustrate the advantages of relational PK-Nets, in the case of transformations between sets of varying cardinalities.

\begin{example}
Let $\mathbf{C}$ be the $T/I$ group, considered as a single-object category, and consider its natural action on the set $\mathbb{Z}_{12}$ of the twelve pitch classes (with the usual semi-tone encoding), which defines a functor $S \colon T/I \to \mathbf{Rel}$.
Let $\Delta$ be the category with three objects $X$, $Y$, and $Z$ and precisely three morphisms $f \colon X \to Y$, $g \colon Y \to Z$, and $h \colon X \to Z$, with $h=gf$. Consider the functor $F \colon \Delta \to T/I$ which sends $f$ to $I_3$, $g$ to $I_5$, and $h$ to $T_2$. 

Consider now a lax functor $R \colon \Delta \to \mathbf{Rel}$ such that we have
\begin{itemize}
\setlength\itemsep{0.05em}
\item{$R(X)=\{x_1,x_2,x_3,x_4\}$, $R(Y)=\{y_1,y_2,y_3\}$, and $R(Z)=\{z_1,z_2,z_3,z_4\}$, and}
\item{the relation $R(f)$ is such that $x_i R(f) y_i$, for $1 \leq i \leq 3$, the relation $R(g)$ is such that $y_i R(f) z_i$, for $1 \leq i \leq 3$, and the relation $R(h)$ is such that $x_i R(f) z_i$, for $1 \leq i \leq 4$.}
\end{itemize}
Consider the left-total lax natural transformation $\phi$ such that
\begin{itemize}
\setlength\itemsep{0.05em}
\item{$\phi_X(x_1)=C$, $\phi_X(x_2)=E$, $\phi_X(x_3)=G$, $\phi_X(x_4)=G$, and}
\item{$\phi_Y(y_1)=E_\flat$, $\phi_Y(y_2)=B$, $\phi_Y(y_3)=G_\sharp$, and}
\item{$\phi_Z(z_1)=D$, $\phi_Z(z_2)=F_\sharp$, $\phi_Z(z_3)=A$, $\phi_Z(z_4)=C$.}
\end{itemize}
Then $(R,S,F,\phi)$ is a relational PK-net of form $R$ and support $S$ which describes the $T_2$ transposition of the dominant seventh $C^7$ chord to the dominant seventh $D^7$ chord and the successive $I_3$ and $I_5$ inversions of its underlying $C$ major triad. 
\end{example}

This simple example is of particular interest as it shows the advantage of relational PK-Nets over the usual PK-Nets in $\mathbf{Sets}$ to describe transformations between sets of decreasing cardinalities. Here, the relation $R(f)$ is a partial function from $R(X)$ to $R(Y)$ which ``forgets'' the element $x_4$, allowing us to describe the transformation of the major triad on which the initial seventh chord is built. The requirement that $R$ be a lax functor appears clearly: the composite relation $R(g) \circ R(f)$ only relates the first three elements $x_1$, $x_2$, and $x_3$ of $R(X)$ to the first three elements $z_1$, $z_2$, and $z_3$ of $R(Z)$ (i.e. the underlying major triads), whereas the relation $R(h)=R(gf)$ relates all four elements, i.e. it describes the full transformation of the seventh chord by the $T_2$ transposition. We thus require that the relation $R(g) \circ R(f)$ be included in $R(h)$.

Relational PK-Nets of form $R$ can be transformed by the mean of PK-homographies, whose definition is as follows.

\begin{definition}
A PK-homography $(N, \nu) \colon K \to K'$ from a relational PK-Net $K= (R, S, F, \phi)$ to a second relational PK-Net $K' = (R, S', F', \phi')$ consists of a functor $N \colon \mathbf{C} \to \mathbf{C'}$ and a left-total lax natural transformation $\nu \colon SF \to S'F'$ such that $F' = NF$ and the composite lax natural transformation $\nu \circ \phi$ is included in $\phi'$. A PK-homography is called a PK-isography if $N$ is an isomorphism and $\nu$ is an equivalence.
\end{definition}

There exists a notion of inclusion of PK-homographies, defined as follows.

\begin{definition}
Let $(N, \nu) \colon K \to K'$ and $(N', \nu') \colon K \to K'$ be PK-homographies from a relational PK-Net $K= (R, S, F, \phi)$ to a second relational PK-Net $K' = (R, S', F', \phi')$. We say that $(N, \nu)$ is included in $(N', \nu')$ if $\nu$ is included in $\nu'$.
\end{definition}

For a given lax functor $R \colon \Delta \to \mathbf{Rel}$, relational PK-Nets with fixed functor $R$ form a 2-category $\mathbf{RelPKN_R}$, which is defined as follows.

\begin{definition}
For a given lax functor $R \colon \Delta \to \mathbf{Rel}$, the 2-category $\mathbf{RelPKN_R}$ has the relational PK-Nets of form $R$ as objects, PK-homographies $(N,\nu)$ between them as 1-morphisms, and inclusion of PK-homographies as 2-morphisms.
\end{definition}

\section{Relational PK-Nets in monoids of parsimonious relations}

In this section, we revisit the work of Douthett about parsimonious relations, and show how we can define proper relational PK-Nets for transformational music analysis on major, minor, and augmented triads. We begin by defining a new monoid of parsimonious relations originating from the Cube Dance.

\subsection{The $M_{\mathcal{UPL}}$ monoid}

As discussed in Section 1.3, the Cube Dance graph presented in Figure \ref{fig:CubeDance} results from the HexaCycles graph to which the vertices and edges corresponding to the four augmented triads and their $\mathcal{P}_{1,0}$ relation to major and minor triads have been added. The HexaCycles graph itself results from the neo-riemannian operations $L$ and $P$ viewed as relations on the set of the 24 major and minor triads. We formalize the construction of the CubeDance graph by defining three different relations on the 28-elements set of the 24 major and minor triads and the four augmented triads. We adopt here the usual semi-tone encoding of pitch-classes with $C=0$, and we notate a major chord by $n_M$ and a minor chord by $n_m$, where $n$ is the root pitch class of the chord, with $0 \leq n \leq 11$. For the four augmented triads, we adopt the following notation: ${A_\flat}_\text{aug}=0_\text{aug}$, ${F}_\text{aug}=1_\text{aug}$, ${D}_\text{aug}=2_\text{aug}$, and ${B}_\text{aug}=3_\text{aug}$. All arithmetic operations are understood modulo 12, unless otherwise indicated.

\begin{definition}
Let $H$ be the set of the 24 major and minor triads and the four augmented triads, i.e. $H=\{n_M, 0 \leq n \leq 11\} \cup \{n_m, 0 \leq n \leq 11\} \cup \{n_\text{aug}, 0 \leq n \leq 3\}$. We define the following relations over $H$.
\begin{itemize}
\setlength\itemsep{0.05em}
\item{The relation $\mathcal{P}$ is the symmetric relation such that we have $n_M \mathcal{P} n_m$ for $0 \leq n \leq 11$, and $n_\text{aug} \mathcal{P} n_\text{aug}$ for $0 \leq n \leq 3$. This is the relational analogue of the neo-Riemannian $P$ operation (not to be confused with the $\mathcal{P}_{m,n}$ relations).}
\item{The relation $\mathcal{L}$ is the symmetric relation such that we have $n_M \mathcal{L} (n+4)_m$ for $0 \leq n \leq 11$, and $n_\text{aug} \mathcal{L} n_\text{aug}$ for $0 \leq n \leq 3$. This is the relational analogue of the neo-Riemannian $L$ operation.}
\item{The relation $\mathcal{U}$ is the symmetric relation such that we have $n_M \mathcal{U} (n \pmod 4)_\text{aug}$ for $0 \leq n \leq 11$, and $n_m \mathcal{U} ((n+3) \pmod 4)_\text{aug}$ for $0 \leq n \leq 11$.}
\end{itemize}
\end{definition}

We are now interested in the monoid $M_{\mathcal{UPL}}$ generated by the relations $\mathcal{U}$, $\mathcal{P}$, and $\mathcal{L}$ under the composition of relations introduced in Section 2.1. The structure of this monoid can be determined by hand through an exhaustive enumeration, or more simply with any computational algebra software, such as GAP \cite{GAP4}.

\begin{proposition}
The monoid $M_{\mathcal{UPL}}$ generated by the relations $\mathcal{U}$, $\mathcal{P}$, and $\mathcal{L}$ has for presentation
$$\begin{aligned} M_{\mathcal{UPL}} =  \langle \mathcal{U}, \mathcal{P}, \mathcal{L} \mid {} & \mathcal{P}^2=\mathcal{L}^2=e, \hspace{0.2cm} \mathcal{LPL}=\mathcal{PLP}, \hspace{0.2cm} \mathcal{U}^3=\mathcal{U}, \\
					&  \mathcal{UP}=\mathcal{UL}, \hspace{0.2cm} \mathcal{PU}=\mathcal{LU}, \hspace{0.2cm} \mathcal{U}^2\mathcal{PU}^2=\mathcal{P}\mathcal{U}^2\mathcal{PU}^2\mathcal{P}, \\
					&  (\mathcal{UP})^2\mathcal{U}^2 = \mathcal{P} (\mathcal{UP})^2\mathcal{U}^2 \mathcal{P}, \hspace{0.2cm}�\mathcal{U}^2(\mathcal{PU})^2 = \mathcal{P} \mathcal{U}^2(\mathcal{PU})^2 \mathcal{P}�\rangle \end{aligned} $$
and contains 40 elements.
\end{proposition}

The Cayley graph of this monoid is presented on Figure \ref{fig:UPLCayley}. The only inversible elements of the $M_{\mathcal{UPL}}$ monoid belong to the set $\{e, \mathcal{L}, \mathcal{P}, \mathcal{LP}, \mathcal{PL}, \mathcal{LPL}\}$, which forms a subgroup in $M_{\mathcal{UPL}}$ isomorphic to the dihedral group $D_6$ generated by the neo-Riemannian operations $P$ and $L$.

In view of building PK-isographies between relational PK-Nets on $M_{\mathcal{UPL}}$, the next proposition establishes the structure of the automorphism group of the $M_{\mathcal{UPL}}$ monoid.

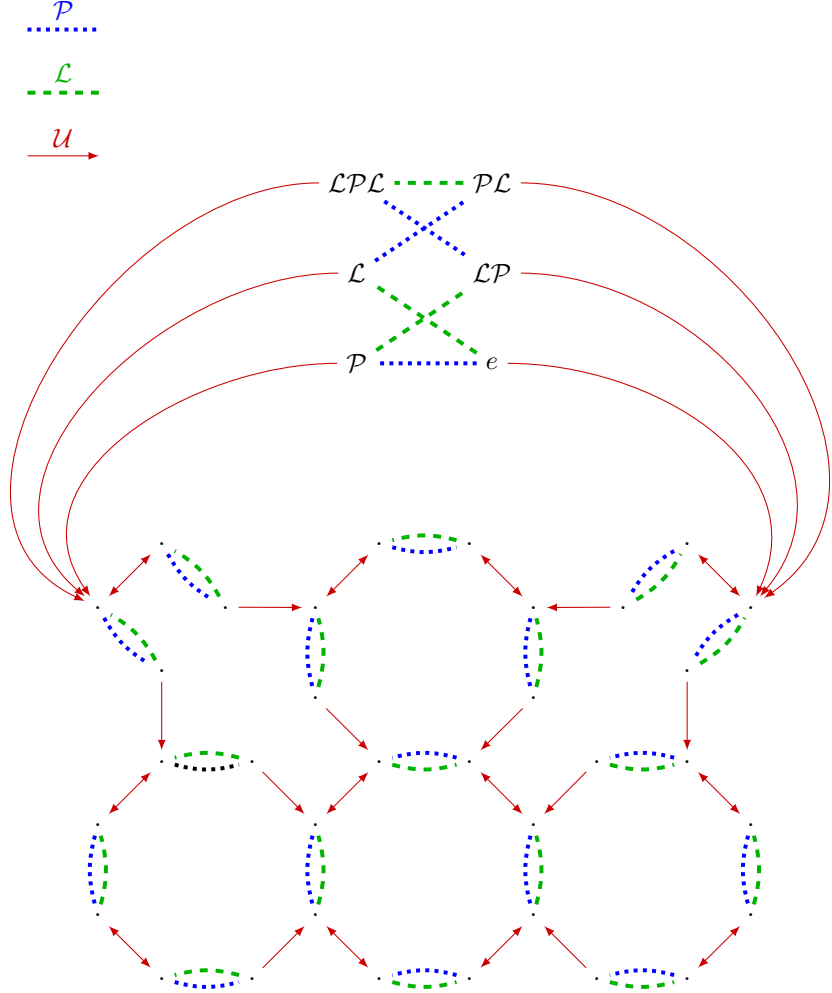
\begin{figure}
\begin{center}
\begin{tikzpicture}[scale=1.2]
	\node (A) at (-4,5.7) {};
	\node (B) at (-3,5.7) {};
	\draw[-,>=latex, dotted,color=blue, line width=1.5] (A) to node[above,midway]{$\mathcal{P}$} (B) ;
	\node (C) at (-4,5) {};
	\node (D) at (-3,5) {};
	\draw[-,>=latex,dashed,color=black!30!green, line width=1.5] (C) to node[above,midway]{$\mathcal{L}$} (D) ;
	\node (E) at (-4,4.3) {};
	\node (F) at (-3,4.3) {};
	\draw[->,>=latex, color=black!20!red] (E) to node[above,midway]{$\mathcal{U}$} (F) ;

	\node (P) at (-0.25,2) {$\mathcal{P}$};
	\node (e) at (1.25,2) {$e$};
	\node (L) at (-0.25,3) {$\mathcal{L}$};
	\node (LP) at (1.25,3) {$\mathcal{LP}$};
	\node (LPL) at (-0.25,4) {$\mathcal{LPL}$};
	\node (PL) at (1.25,4) {$\mathcal{PL}$};
	\draw[-,>=latex, dotted, color=blue, line width=1.5] (e) to (P) ;
	\draw[-,>=latex, dotted, color=blue, line width=1.5] (L) to (PL) ;
	\draw[-,>=latex, dotted, color=blue, line width=1.5] (LP) to (LPL) ;
	\draw[-,>=latex, dashed, color=black!30!green, line width=1.5] (e) to (L) ;
	\draw[-,>=latex, dashed, color=black!30!green, line width=1.5] (P) to (LP) ;
	\draw[-,>=latex, dashed, color=black!30!green, line width=1.5] (PL) to (LPL) ;

	\node (A1) at (0,0) {$.$};
	\node (A2) at (1,0) {$.$};
	\node (A3) at (1.707,-0.707) {$.$};
	\node (A4) at (1.707,-1.707) {$.$};
	\node (A5) at (-0.707,-1.707) {$.$};
	\node (A6) at (-0.707,-0.707) {$.$};
	\node (A7) at (1,-2.41) {$.$};
	\node (A8) at (0,-2.41) {$.$};

	\node (A9) at (1.707,-0.707-2.41) {$.$};
	\node (A10) at (1.707,-1.707-2.41) {$.$};
	\node (A11) at (1,-2.41-2.41) {$.$};
	\node (A12) at (0,-2.41-2.41) {$.$};
	\node (A13) at (-0.707,-1.707-2.41) {$.$};
	\node (A14) at (-0.707,-0.707-2.41) {$.$};

	\node (A15) at (1+2.41,-2.41) {$.$};
	\node (A16) at (0+2.41,-2.41) {$.$};
	\node (A17) at (1.707+2.41,-0.707-2.41) {$.$};
	\node (A18) at (1.707+2.41,-1.707-2.41) {$.$};
	\node (A19) at (1+2.41,-2.41-2.41) {$.$};
	\node (A20) at (0+2.41,-2.41-2.41) {$.$};

	\node (A21) at (1-2.41,-2.41) {$.$};
	\node (A22) at (0-2.41,-2.41) {$.$};
	\node (A23) at (1-2.41,-2.41-2.41) {$.$};
	\node (A24) at (0-2.41,-2.41-2.41) {$.$};
	\node (A25) at (-0.707-2.41,-1.707-2.41) {$.$};
	\node (A26) at (-0.707-2.41,-0.707-2.41) {$.$};	

	\node (A27) at (1+2.41,-2.41+1.0) {$.$};
	\node (U) at (1+2.41+0.707,-2.41+0.707+1.0) {$.$};
	\node (U2) at (1+2.41,-2.41+0.707+1.0+0.707) {$.$};
	\node (A30) at (1+2.41-0.707,-2.41+0.707+1.0) {$.$};

	\node (A31) at (-2.41,-2.41+1.0) {$.$};
	\node (A32) at (-2.41+0.707,-2.41+0.707+1.0) {$.$};
	\node (U2P) at (-2.41,-2.41+0.707+1.0+0.707) {$.$};
	\node (UP) at (-2.41-0.707,-2.41+0.707+1.0) {$.$};

	\draw[->,>=latex,black!20!red] (e) to [out=0,in=65,looseness=1.1] (U) ;
	\draw[->,>=latex,black!20!red] (LP) to [out=0,in=50,looseness=1.1] (U) ;
	\draw[->,>=latex,black!20!red] (PL) to [out=0,in=35,looseness=1.1] (U) ;
	\draw[->,>=latex,black!20!red] (P) to [out=180,in=125,looseness=1.1] (UP) ;
	\draw[->,>=latex,black!20!red] (L) to [out=180,in= 140,looseness=1.1] (UP) ;
	\draw[->,>=latex,black!20!red] (LPL) to [out=180,in= 155,looseness=1.1] (UP) ;
	
	\draw[<->,>=latex,black!20!red] (U) to (U2) ;
	\draw[<->,>=latex,black!20!red] (UP) to (U2P) ;
	\draw[->,>=latex,black!20!red] (A27) to (A15) ;
	\draw[->,>=latex,black!20!red] (A31) to (A22) ;
	\draw[->,>=latex,black!20!red] (A32) to (A6) ;
	\draw[->,>=latex,black!20!red] (A30) to (A3) ;
	\draw[->,>=latex,black!20!red] (A5) to (A8) ;
	\draw[->,>=latex,black!20!red] (A21) to (A14) ;
	\draw[->,>=latex,black!20!red] (A4) to (A7) ;
	\draw[->,>=latex,black!20!red] (A16) to (A9) ;
	\draw[->,>=latex,black!20!red] (A20) to (A10) ;
	\draw[->,>=latex,black!20!red] (A23) to (A13) ;
	\draw[<->,>=latex,black!20!red] (A22) to (A26) ;
	\draw[<->,>=latex,black!20!red] (A24) to (A25) ;
	\draw[<->,>=latex,black!20!red] (A15) to (A17) ;
	\draw[<->,>=latex,black!20!red] (A18) to (A19) ;
	\draw[<->,>=latex,black!20!red] (A10) to (A11) ;
	\draw[<->,>=latex,black!20!red] (A12) to (A13) ;
	\draw[<->,>=latex,black!20!red] (A8) to (A14) ;
	\draw[<->,>=latex,black!20!red] (A7) to (A9) ;
	\draw[<->,>=latex,black!20!red] (A1) to (A6) ;
	\draw[<->,>=latex,black!20!red] (A2) to (A3) ;
	
	\draw[-,>=latex,dotted, color=blue, line width=1.5] (U2) to[bend right=15] (A30) ;
	\draw[-,>=latex,dashed, color=black!30!green, line width=1.5] (A30) to[bend right=15] (U2) ;

	\draw[-,>=latex,dotted, color=blue, line width=1.5] (A1) to[bend right=15] (A2) ;
	\draw[-,>=latex,dashed, color=black!30!green, line width=1.5] (A2) to[bend right=15] (A1) ;

	\draw[-,>=latex,dotted, color=blue, line width=1.5] (A6) to[bend right=15] (A5) ;
	\draw[-,>=latex,dashed, color=black!30!green, line width=1.5] (A5) to[bend right=15] (A6) ;

	\draw[-,>=latex,dotted, color=blue, line width=1.5] (A3) to[bend right=15] (A4) ;
	\draw[-,>=latex,dashed, color=black!30!green, line width=1.5] (A4) to[bend right=15] (A3) ;

	\draw[-,>=latex,dotted, color=blue, line width=1.5] (A7) to[bend right=15] (A8) ;
	\draw[-,>=latex,dashed, color=black!30!green, line width=1.5] (A8) to[bend right=15] (A7) ;

	\draw[-,>=latex,dotted, line width=1.5] (A22) to[bend right=15] (A21) ;
	\draw[-,>=latex,dashed, color=black!30!green, line width=1.5] (A21) to[bend right=15] (A22) ;

	\draw[-,>=latex,dotted, color=blue, line width=1.5] (A26) to[bend right=15] (A25) ;
	\draw[-,>=latex,dashed, color=black!30!green, line width=1.5] (A25) to[bend right=15] (A26) ;

	\draw[-,>=latex,dotted, color=blue, line width=1.5] (A14) to[bend right=15] (A13) ;
	\draw[-,>=latex,dashed, color=black!30!green, line width=1.5] (A13) to[bend right=15] (A14) ;

	\draw[-,>=latex,dotted, color=blue, line width=1.5] (A24) to[bend right=15] (A23) ;
	\draw[-,>=latex,dashed, color=black!30!green, line width=1.5] (A23) to[bend right=15] (A24) ;

	\draw[-,>=latex,dotted, color=blue, line width=1.5] (A11) to[bend right=15] (A12) ;
	\draw[-,>=latex,dashed, color=black!30!green, line width=1.5] (A12) to[bend right=15] (A11) ;

	\draw[-,>=latex,dotted, color=blue, line width=1.5] (A9) to[bend right=15] (A10) ;
	\draw[-,>=latex,dashed, color=black!30!green, line width=1.5] (A10) to[bend right=15] (A9) ;

	\draw[-,>=latex,dotted, color=blue, line width=1.5] (A17) to[bend right=15] (A18) ;
	\draw[-,>=latex,dashed, color=black!30!green, line width=1.5] (A18) to[bend right=15] (A17) ;

	\draw[-,>=latex,dotted, color=blue, line width=1.5] (A19) to[bend right=15] (A20) ;
	\draw[-,>=latex,dashed, color=black!30!green, line width=1.5] (A20) to[bend right=15] (A19) ;

	\draw[-,>=latex,dotted, color=blue, line width=1.5] (A15) to[bend right=15] (A16) ;
	\draw[-,>=latex,dashed, color=black!30!green, line width=1.5] (A16) to[bend right=15] (A15) ;

	\draw[-,>=latex,dotted, color=blue, line width=1.5] (U) to[bend right=15] (A27) ;
	\draw[-,>=latex,dashed, color=black!30!green, line width=1.5] (A27) to[bend right=15] (U) ;

	\draw[-,>=latex,dotted, color=blue, line width=1.5] (UP) to[bend right=15] (A31) ;
	\draw[-,>=latex,dashed, color=black!30!green, line width=1.5] (A31) to[bend right=15] (UP) ;

	\draw[-,>=latex,dotted, color=blue, line width=1.5] (U2P) to[bend right=15] (A32) ;
	\draw[-,>=latex,dashed, color=black!30!green, line width=1.5] (A32) to[bend right=15] (U2P) ;
	
\end{tikzpicture}
\end{center}
\caption{The Cayley graph of the monoid $M_{\mathcal{UPL}}$ generated by the relations $\mathcal{U}$, $\mathcal{P}$, and $\mathcal{L}$. The relations $\mathcal{L}$ and $\mathcal{P}$ are involutions, and are represented as arrowless dashed and dotted lines.}
\label{fig:UPLCayley}
\end{figure}

\begin{proposition}
The automorphism group of the $M_{\mathcal{UPL}}$ monoid is isomorphic to $D_6 \times \mathbb{Z}_2$.
\end{proposition}
\begin{proof}
Any automorphism $N$ of the $M_{\mathcal{UPL}}$ monoid is entirely determined by the image of its generators. Since $\mathcal{L}$ and $\mathcal{P}$ are the only inversible generators, their images belong to the $D_6$ subgroup $\{e, \mathcal{L}, \mathcal{P}, \mathcal{LP}, \mathcal{PL}, \mathcal{LPL}\}$ and induce an isomorphism of this subgroup. From known results about dihedral groups, we have $\text{Aut}(D_6) \simeq D_6$, and the images are given by $N( \mathcal{P} ) = ( \mathcal{PL} )^{m+n} \mathcal{L}$, and $N( \mathcal{L} ) = ( \mathcal{PL} )^n \mathcal{L}$, with $m \in \{-1,1\}$ and $n \in \{0,1,2\}$.

This defines a homomorphism $\Phi$ from $\text{Aut}(M_{\mathcal{UPL}})$ to $D_6$ which associates to any automorphism $N$ of $M_{\mathcal{UPL}}$ the automorphism of $D_6$ induced by the images of $\mathcal{L}$ and $\mathcal{P}$.

The kernel of $\Phi$ consists of the subgroup of $\text{Aut}(M_{\mathcal{UPL}})$ formed by the automorphisms $N$ such that $N(\mathcal{L})=\mathcal{L}$ and $N(\mathcal{P})=\mathcal{P}$. It is thus uniquely determined by the possible images of the remaining generator $\mathcal{U}$ by $N$. An exhaustive computer search shows that only $N(\mathcal{U})=\mathcal{U}$ and $N(\mathcal{U})=\mathcal{PUP}$ yield valid automorphisms, i.e. $N(\mathcal{U}) = \mathcal{P}^k\mathcal{U}\mathcal{P}^k$, with $k \in \{0,1\}$ considered as the additive cyclic group $\mathbb{Z}_2$.
Thus $\text{Aut}(M_{\mathcal{UPL}})$ is an extension of $\mathbb{Z}_2$ by $\mathbb{D}_6$, and any automorphism $N$ is uniquely determined by the pair $(g,k)$ where $g$ is an element of $\text{Aut}(D_6)$, and $k$ is an element of $\mathbb{Z}_2$. 

Let $N_1=(g_1,k_1)$ and $N_2=(g_2,k_2)$ be two automorphisms of the $M_{\mathcal{UPL}}$ monoid and consider $N=N_2N_1=(g,k)$. From the discussion above, we have $g=g_2g_1$. The image of the generator $\mathcal{U}$ by $N$ is
$$N(\mathcal{U})=N_2(\mathcal{P}^{k_1}\mathcal{U}\mathcal{P}^{k_1}),$$
which is equal to
$$N(\mathcal{U})=((\mathcal{PL})^{m_2+n_2}\mathcal{L})^{k_1}\mathcal{P}^{k_2}\mathcal{U}\mathcal{P}^{k_2}((\mathcal{PL})^{m_2+n_2}\mathcal{L})^{k_1}.$$
Since we have $\mathcal{PU}=\mathcal{LU}$ and $\mathcal{UP}=\mathcal{UL}$, all the terms $\mathcal{P}$ in this last equation can be replaced by $\mathcal{L}$, and since $\mathcal{L}^2=e$, this yields
$$N(\mathcal{U})=\mathcal{L}^{k_1+k_2}\mathcal{U}\mathcal{L}^{k_1+k_2}.$$
Hence, we have $N=N_2N_1=(g,k)=(g_2g_1,k_2+k_1)$, thus proving that $\text{Aut}(M_{\mathcal{UPL}})$ is isomorphic to $D_6 \times \mathbb{Z}_2$.
\end{proof}

\begin{figure}
\centering

\subfigure[]{
\includegraphics[scale=0.65]{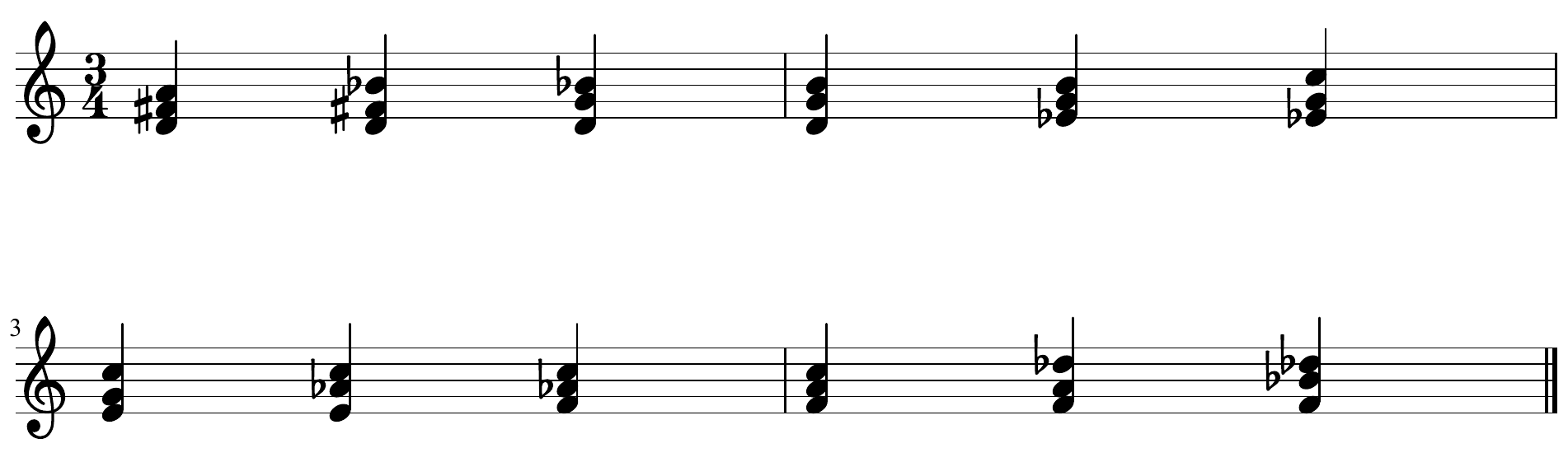}
\label{subfig:Muse-TAB}
}

\vspace{1.0cm}
\subfigure[]{
\begin{tikzpicture}[scale=1]
	\node (A) at (-6,0) {$D_M$};
	\node (B) at (-4,0) {$D_\text{aug}$};
	\node (C) at (-2,0) {$G_m$};
	\draw[->,>=latex] (A) to node[above,midway]{$\mathcal{U}$} (B) ;
	\draw[->,>=latex] (B) to node[above,midway]{$\mathcal{U}$} (C) ;
	
	\node (D) at (0,0) {$G_M$};
	\node (E) at (2,0) {$B_\text{aug}$};
	\node (F) at (4,0) {$C_m$};

	\draw[->,>=latex] (C) to node[above,midway]{$\mathcal{P}$} (D) ;
	\draw[->,>=latex] (D) to node[above,midway]{$\mathcal{U}$} (E) ;
	\draw[->,>=latex] (E) to node[above,midway]{$\mathcal{U}$} (F) ;

	\node (G) at (-6,-2) {$C_M$};
	\node (H) at (-4,-2) {$C_\text{aug}$};
	\node (I) at (-2,-2) {$F_m$};

	\draw[->,>=latex] (F) to node[below,midway]{$\mathcal{P}$} (G) ;
	\draw[->,>=latex] (G) to node[below,midway]{$\mathcal{U}$} (H) ;
	\draw[->,>=latex] (H) to node[below,midway]{$\mathcal{U}$} (I) ;

	\node (J) at (0,-2) {$F_M$};
	\node (K) at (2,-2) {$F_\text{aug}$};
	\node (L) at (4,-2) {${B_\flat}_m$};

	\draw[->,>=latex] (I) to node[below,midway]{$\mathcal{P}$} (J) ;
	\draw[->,>=latex] (J) to node[below,midway]{$\mathcal{U}$} (K) ;
	\draw[->,>=latex] (K) to node[below,midway]{$\mathcal{U}$} (L) ;
\end{tikzpicture}
\label{subfig:Muse-TAB-Analysis1}
} 

\vspace{1.0cm}
\subfigure[]{
\begin{tikzpicture}[scale=1]
	\node (A) at (-8,0) {$D_M$};
	\node (B) at (-7,-1.3) {$D_\text{aug}$};
	\node (C) at (-8,-2.3) {$G_m$};
	\draw[->,>=latex] (A) to node[shift={(0.3,0.15)}]{$\mathcal{U}$} (B) ;
	\draw[->,>=latex] (B) to node[shift={(0.3,-0.15)}]{$\mathcal{U}$} (C) ;
	\draw[->,>=latex] (A) to node[left,midway]{$\mathcal{U}^2$} (C) ;
	
	\node (D) at (-4,0) {$G_M$};
	\node (E) at (-3,-1.3) {$B_\text{aug}$};
	\node (F) at (-4,-2.3) {$C_m$};

	\draw[->,>=latex] (D) to node[shift={(0.3,0.15)}]{$\mathcal{U}$} (E) ;
	\draw[->,>=latex] (E) to node[shift={(0.3,-0.15)}]{$\mathcal{U}$} (F) ;
	\draw[->,>=latex] (D) to node[left,midway]{$\mathcal{U}^2$} (F) ;
	
	\node (T1) at (-6.5,-1.3) {};
	\node (T2) at (-4.75,-1.3) {};
	\draw[->,>=latex] (T1) to node[above,midway] {$(N=Id,F\nu)$}  (T2) ;

	\node (G) at (0,0) {$C_M$};
	\node (H) at (1,-1.3) {$C_\text{aug}$};
	\node (I) at (0,-2.3) {$F_m$};

	\draw[->,>=latex] (G) to node[shift={(0.3,0.15)}]{$\mathcal{U}$} (H) ;
	\draw[->,>=latex] (H) to node[shift={(0.3,-0.15)}]{$\mathcal{U}$} (I) ;
	\draw[->,>=latex] (G) to node[left,midway]{$\mathcal{U}^2$} (I) ;

	\node (T3) at (-2.5,-1.3) {};
	\node (T4) at (-0.75,-1.3) {};
	\draw[->,>=latex] (T3) to node[above,midway] {$(N=Id,F\nu)$}  (T4);

	\node (J) at (4,0) {$F_M$};
	\node (K) at (5,-1.3) {$F_\text{aug}$};
	\node (L) at (4,-2.3) {${B_\flat}_m$};

	\draw[->,>=latex] (J) to node[shift={(0.3,0.15)}]{$\mathcal{U}$} (K) ;
	\draw[->,>=latex] (K) to node[shift={(0.3,-0.15)}]{$\mathcal{U}$} (L) ;
	\draw[->,>=latex] (J) to node[left,midway]{$\mathcal{U}^2$} (L) ;

	\node (T5) at (1.5,-1.3) {};
	\node (T6) at (3.25,-1.3) {};
	\draw[->,>=latex] (T5) to node[above,midway] {$(N=Id,F\nu)$} (T6) ;
\end{tikzpicture}
\label{subfig:Muse-TAB-Analysis2}
} 

\caption{\subref{subfig:Muse-TAB} Reduction of the opening progression of \textit{Take A Bow} from Muse (the first twelve chords are represented here). \subref{subfig:Muse-TAB-Analysis1} First transformational analysis in the $M_{\mathcal{UPL}}$ monoid showing the sequential regularity of the progression. \subref{subfig:Muse-TAB-Analysis1} Second transformational analysis in the $M_{\mathcal{UPL}}$ monoid showing the successive transformations of the initial three-chord cell by the homography $(N=Id,F\nu)$ with $\nu(n_M)=(n+5)_M$, $\nu(n_m)=(n+5)_m$, and $\nu(n_\text{aug})=(n+1 \pmod{4})_\text{aug}$.}
\label{fig:Muse-TAB}
\end{figure}

We now illustrate the possibilities offered by the $M_{\mathcal{UPL}}$ monoid for transformational analysis using relational PK-Nets.
Figure \ref{subfig:Muse-TAB} shows a reduction of the opening chord progression of the song \textit{Take A Bow} by the English rock band Muse. This progression proceeds by semitone changes from a major chord to an augmented chord to a minor chord. At this point, the minor chord evolves to a major chord on the same root, i.e. the two chords are related by the neo-Riemannian operation $P$. The same process is then applied five times till the middle of the song (only the first twelve chords are presented in Figure \ref{subfig:Muse-TAB}).

A first transformational analysis of this progression can be realized as follows. We focus here on the first four chords, since the progression further proceeds identically. The data of the monoid $M_{\mathcal{UPL}}$ and the relations over $H$ of its elements defines a functor $S \colon M_{\mathcal{UPL}} \to \mathbf{Rel}$. Let $\Delta$ define the order of the ordinal number $\mathbf{4}$ (whose objects are labelled $X_i$, with $0 \leq i \leq 3$), and let $R$ be the functor from $\Delta$ to $\mathbf{Rel}$ which sends the objects $X_i$ of $\Delta$ to singletons $\{x_i\}$. Let $F$ be the functor from $\Delta$ to $M_{\mathcal{UPL}}$ which sends the non-trivial morphisms $f_{0,1} \colon X_0 \to X_1$ and $f_{1,2} \colon X_1 \to X_2$ of $\Delta$ to $\mathcal{U}$ in $\mathcal{UPL}$, and the non-trivial morphism $f_{2,3} \colon X_2 \to X_3$ of $\Delta$ to $\mathcal{P}$ in $\mathcal{UPL}$.
Finally, let $\phi$ be the left-total lax natural transformation which sends $x_0$ to $D_M$, $x_1$ to $D_\text{aug}$, $x_2$ to $G_m$, and $x_3$ to $G_M$. Then $(R,S,F,\phi)$ is a relational PK-Net describing the opening progression of Figure \subref{subfig:Muse-TAB}. This PK-Net (and its extension to the remaining chords) is shown in a simplified way, by directly labelling the arrows between the chords with the elements of $M_{\mathcal{UPL}}$. Given two chords on the left and right side of a labelled arrow between them, one should remain aware that this does not mean that the chord on the right is the unique image of the chord by the given element of $M_{\mathcal{UPL}}$. From a relational point-of-view, as discussed in Section 2.2, one should consider instead the chord on the right to be \textit{related} to the one on the left by the element of $M_{\mathcal{UPL}}$, among the possibly larger range of possibilities given by the functor $S \colon M_{\mathcal{UPL}} \to \mathbf{Rel}$. For example, given the chord $D_\text{aug}$, there exists six chords $y$ (namely $D_M$, ${F_\sharp}_M$, ${B_\flat}_M$, ${E_\flat}_m$, $B_m$, and $G_m$) such that we have $D_\text{aug} \mathcal{U} y$. Here, the left-total lax natural transformation $\phi$ allows us to select precisely one chord, $G_m$, to explain the given chord progression.

This relational PK-Net shows the regularity of the chord progression, but does not clearly evidence the progression by fifths of the initial three-chord cell. We propose now a second transformational analysis based on a specific PK-Net isography and its iterated application on a relational PK-Net describing this initial cell.

Let $\Delta$ define the order of the ordinal number $\mathbf{3}$ (whose objects are labelled $X_i$, with $0 \leq i \leq 2$), and let $R$ be the functor from $\Delta$ to $\mathbf{Rel}$ which sends the objects $X_i$ of $\Delta$ to singletons $\{x_i\}$. Let $F$ be the functor from $\Delta$ to $M_{\mathcal{UPL}}$ which sends the non-trivial morphisms $f_{0,1} \colon X_0 \to X_1$ and $f_{1,2} \colon X_1 \to X_2$ of $\Delta$ to $\mathcal{U}$ in $\mathcal{UPL}$. Finally, let $\phi$ be the left-total lax natural transformation which sends $x_0$ to $D_M$, $x_1$ to $D_\text{aug}$, $x_2$ to $G_m$. Then $(R,S,F,\phi)$ is a relational PK-Net describing the initial three-chord cell of Figure \subref{subfig:Muse-TAB}.
Consider now the identity functor $N=Id$ on $M_{\mathcal{UPL}}$, along with the left-total lax natural transformation $\nu \colon S \to S$ defined on the set of the major, minor, and augmented triads by $\nu(n_M)=(n+5)_M$, $\nu(n_m)=(n+5)_m$, and $\nu(n_\text{aug})=(n+1 \pmod{4})_\text{aug}$. By applying the PK-isography $(N,F\nu)$ on $(R,S,F,\phi)$, one obtains a second relational PK-Net $(R,S,F,\phi')$ such that $\phi'$ sends $x_0$ to $G_M$, $x_1$ to $B_\text{aug}$, $x_2$ to $C_m$, thus describing the progression by fifth of the initial cell. The successive chords are given by the iterated application of the same PK-isography $(N,F\nu)$.

\subsection{The $M_\mathcal{S}$, $M_\mathcal{T}$, and $M_{\mathcal{ST}}$ monoids}

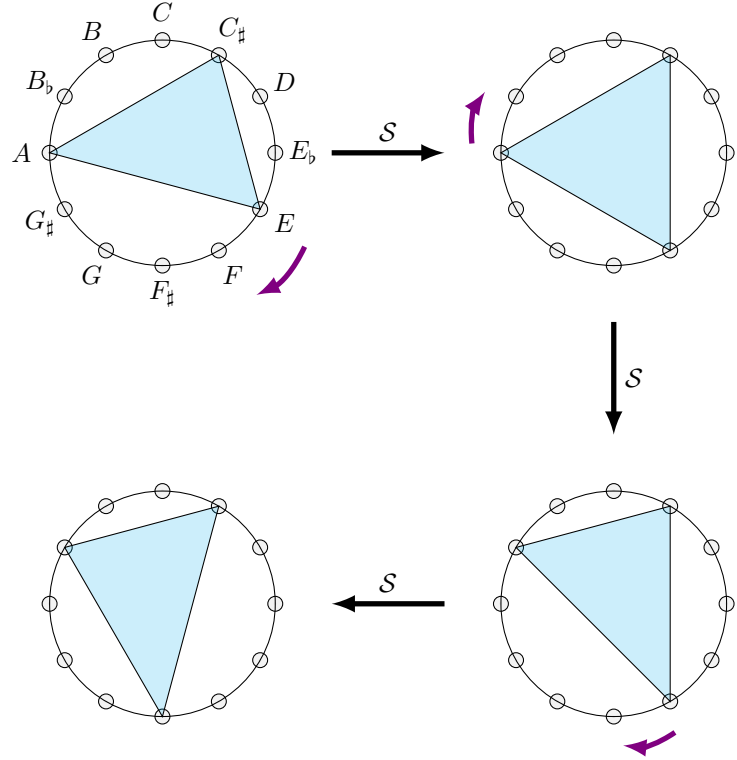
\begin{figure}[h!]
\begin{center}
\begin{tikzpicture}[scale=0.75]
	\draw (-5,0) circle[radius=2.0]; 
	\foreach \i in {0,...,11} {
		\node[circle,draw=black,fill=black, fill opacity = 0.05, inner sep=1pt, minimum size=3pt] (A\i) at ({-5.0+2.0*sin(\i*30)},{2.0*cos(\i*30)}) {.};
	}
	\foreach \i in {0,...,11} {
		\node (B\i) at ({-5.0+3.0*sin(\i*30)},{3.0*cos(\i*30)}) {};
	}
	\foreach \i / \name in {0/$C$,1/$C_\sharp$,2/$D$,3/$E_\flat$,4/$E$,5/$F$,6/$F_\sharp$,7/$G$,8/$G_\sharp$,9/$A$,10/$B_\flat$,11/$B$} {
		\node (N\i) at ({-5.0+2.5*sin(\i*30)},{2.5*cos(\i*30)}) {\name};
	}
	\draw[ draw=black, fill=cyan, fill opacity=0.2] (A9.center) -- (A1.center) -- (A4.center) --cycle;
	\draw[ -latex, line width=2, color=violet] (B4) to[bend left=20] (B5); 

	\draw[->,>=latex, color=black,line width=2,] (-2.0,0.0) to node[above,midway]{$\mathcal{S}$} (0,0) ;

	\draw (3,0) circle[radius=2.0]; 
	\foreach \i in {0,...,11} {
		\node[circle,draw=black,fill=black, fill opacity = 0.05, inner sep=1pt, minimum size=3pt] (C\i) at ({3+2.0*sin(\i*30)},{2.0*cos(\i*30)}) {.};
	}
	\foreach \i in {0,...,11} {
		\node (D\i) at ({3+2.5*sin(\i*30)},{2.5*cos(\i*30)}) {};
	}
	\draw[ draw=black, fill=cyan, fill opacity=0.2] (C9.center) -- (C1.center) -- (C5.center) --cycle;
	\draw[ -latex, line width=2, color=violet] (D9) to[bend left=20] (D10); 

	\draw[->,>=latex, color=black,line width=2,] (3.0,-3.0) to node[right,midway]{$\mathcal{S}$} (3.0,-5.0) ;

	\draw (3.0,-8) circle[radius=2.0]; 
	\foreach \i in {0,...,11} {
		\node[circle,draw=black,fill=black, fill opacity = 0.05, inner sep=1pt, minimum size=3pt] (E\i) at ({3.0+2.0*sin(\i*30)},{-8+2.0*cos(\i*30)}) {.};
	}
	\foreach \i in {0,...,11} {
		\node (F\i) at ({3.0+2.5*sin(\i*30)},{-8+2.5*cos(\i*30)}) {};
	}
	\draw[ draw=black, fill=cyan, fill opacity=0.2] (E5.center) -- (E10.center) -- (E1.center) --cycle;
	\draw[ -latex, line width=2, color=violet] (F5) to[bend left=20] (F6); 

	\draw[->,>=latex, color=black,line width=2,] (0.0,-8.0) to node[above,midway]{$\mathcal{S}$} (-2.0,-8.0) ;

	\draw (-5.0,-8) circle[radius=2.0]; 
	\foreach \i in {0,...,11} {
		\node[circle,draw=black,fill=black, fill opacity = 0.05, inner sep=1pt, minimum size=3pt] (G\i) at ({-5.0+2.0*sin(\i*30)},{-8+2.0*cos(\i*30)}) {.};
	}
	\foreach \i in {0,...,11} {
		\node (H\i) at ({-5.0+2.5*sin(\i*30)},{-8+2.5*cos(\i*30)}) {};
	}
	\draw[ draw=black, fill=cyan, fill opacity=0.2] (G6.center) -- (G10.center) -- (G1.center) --cycle;
\end{tikzpicture}
\end{center}
\caption{A possible path for the progression from $A$ major to $F_\sharp$ major by successive semitone displacement. The labels above the arrows indicate that the chord on the right is related to the chord on the left by the $\mathcal{S}$ relation. The corresponding semitone displacements are indicated by the curved violet arrows.}
\label{fig:S-Progression}
\end{figure}

In the previous monoid of parsimonious relations, we distinguished the subrelations $\mathcal{U}$, $\mathcal{P}$, and $\mathcal{L}$ included in the $\mathcal{P}_{1,0}$ relation in order to differentiate the contributions of the neo-Riemannian operations $P$ and $L$, and the relation $\mathcal{U}$ which bridges the hexatonic system \textit{via} the augmented triads. However, we could also consider Douthett's $\mathcal{P}_{1,0}$ relation as a whole, to focus on the parsimonious voice-leading between chords by semitone displacement.

For clarity of notation, we rename the $\mathcal{P}_{1,0}$ relation as $\mathcal{S}$ and we recall its definition on the set of major, minor, and augmented triads.

\begin{definition}
Let $H$ be the set of the 24 major and minor triads and the four augmented triads, i.e. $H=\{n_M, 0 \leq n \leq 11\} \cup \{n_m, 0 \leq n \leq 11\} \cup \{n_\text{aug}, 0 \leq n \leq 3\}$. The $\mathcal{S}$ relation over $H$ is defined as the symmetric relation such that we have
\begin{itemize}
\setlength\itemsep{0.05em}
\item{$n_M \mathcal{S} n_m$,  $n_M \mathcal{S} (n+4)_m$, and $n_M \mathcal{S} (n \pmod{4})_\text{aug}$ for $0 \leq n \leq 11$, and}
\item{$n_m \mathcal{S} n_M$,  $n_m \mathcal{S} (n+8)_M$, and $n_m \mathcal{S} ((n+3) \pmod{4})_\text{aug}$ for $0 \leq n \leq 11$.}
\end{itemize}
\end{definition}

We are now interested in the monoid generated by the relation $\mathcal{S}$, under the composition of relations introduced in Section 2.1. The structure of this monoid can easily be determined by hand or with a computer.

\begin{proposition}
The monoid $M_\mathcal{S}$ generated by the relation $\mathcal{S}$ has for presentation
$$M_\mathcal{S} = \langle \mathcal{S} \mid \mathcal{S}^7= \mathcal{S}^5 \rangle.$$
\end{proposition}

As a quick application of this monoid, consider the following example. The data of the monoid $M_{\mathcal{S}}$ and the relations over $H$ of its elements defines a functor $S \colon M_{\mathcal{S}} \to \mathbf{Rel}$. Let $\Delta$ define the order of the ordinal number $\mathbf{2}$, i.e. the category with only two objects $X$ and $Y$ and only one non-trivial morphism $F \colon X \to Y$ between them, and let $R$ be the functor from $\Delta$ to $\mathbf{Rel}$ which sends the objects $X$ and $Y$ of $\Delta$ to the singletons $\{x\}$ and $\{y\}$.
Let $\mathbf{C}$ be the $M_{\mathcal{P}_{1,0}}$-monoid with the above-defined functor $S \colon M_\mathcal{S} \to \mathbf{Rel}$.
Finally, let $\phi$ be the left-total lax natural transformation which sends $x$ to $A_M$, and $y$ to ${F_\sharp}_M$.
We are interested in the possible functors $F$ such that $(R,S,F,\phi)$ is a relational PK-Net describing the relation between $A_M$ and ${F_\sharp}_M$.
A rapid verification through the elements of $M_\mathcal{S}$ yields that only $F(f)=\mathcal{S}^3$ and $F(f)=\mathcal{S}^5$ yield valid choices for $F$.
Observe that these relations correspond to the first two shortest distances between $A_M$ and ${F_\sharp}_M$ in the Cube Dance of Figure \ref{fig:CubeDance}.
A possible path for the progression from $A_M$ to ${F_\sharp}_M$ by three successive semitone displacements is given in Figure \ref{fig:S-Progression}, which shows the successive relations $A_M \mathcal{S} F_\text{aug}$, $F_\text{aug} \mathcal{S} {B_\flat}_m$, and ${B_\flat}_m \mathcal{S} {F_\sharp}_m$.

As introduced in Section 1.3, Douthett also studied the parsimonious graph induced by the $\mathcal{P}_{2,0}$ relation on the set of major, minor, and augmented triads (the ``Weitzmann's Waltz'' graph shown in Figure \ref{fig:WeitzmannWaltz}). The $\mathcal{P}_{2,0}$ relation relates two pitch-class sets if one can be obtained from the other by the displacement of two pitch-classes by a semitone each, which includes both the case of the parallel displacement of these pitch classes, as well as their contrary movement. 

As before, we rename the $\mathcal{P}_{2,0}$ relation as $\mathcal{T}$ for clarity of notation, and we recall its specific definition on the set of major, minor, and augmented triads. 

\begin{definition}
Let $H$ be the set of the 24 major and minor triads and the four augmented triads, i.e. $H=\{n_M, 0 \leq n \leq 11\} \cup \{n_m, 0 \leq n \leq 11\} \cup \{n_\text{aug}, 0 \leq n \leq 3\}$. The $\mathcal{T}$ relation over $H$ is defined as the symmetric relation such that we have
\begin{itemize}
\setlength\itemsep{0.05em}
\item{$n_M \mathcal{T} (n+4)_M$, $n_M \mathcal{T} (n+8)_M$, $n_M \mathcal{T} (n+1)_m$,  $n_M \mathcal{T} (n+5)_m$, and $n_M \mathcal{T} ((n+3) \pmod{4})_\text{aug}$ for $0 \leq n \leq 11$, and}
\item{$n_m \mathcal{T} (n+4)_m$, $n_m \mathcal{T} (n+8)_m$, $n_m \mathcal{T} (n+11)_M$,  $n_m \mathcal{T} (n+7)_M$, and $n_m \mathcal{T} (n \pmod{4})_\text{aug}$ for $0 \leq n \leq 11$.}
\end{itemize}
\end{definition}

As before, the structure of the monoid $M_\mathcal{T}$ generated by the relation $\mathcal{T}$ can easily be determined.

\begin{proposition}
The monoid $M_\mathcal{T}$ generated by the relation $\mathcal{T}$ has for presentation
$$M_\mathcal{T} = \langle \mathcal{T} \mid \mathcal{T}^4= \mathcal{T}^3 \rangle.$$
\end{proposition}

\begin{figure}
\centering
\begin{center}
\begin{tikzpicture}[scale=1.1]
	\node (A) at (-4,1) {};
	\node (B) at (-3,1) {};
	\draw[->,>=latex] (A) to node[above,midway]{$\mathcal{S}$} (B) ;
	\node (C) at (-4,0.3) {};
	\node (D) at (-3,0.3) {};
	\draw[->,>=latex,dashed] (C) to node[above,midway]{$\mathcal{T}$} (D) ;

	\node (e) at (0,0) {$e$};
	\node (T) at (4,-2) {$\mathcal{T}$};
	\draw[->,>=latex,dashed] (e) to (T) ;
	\node (S) at (-4,-2) {$\mathcal{S}$};
	\draw[->,>=latex] (e) to (S) ;
	\node (TS) at (0,-2) {$\mathcal{TS}$};
	\draw[->,>=latex] (T) to (TS) ;
	\draw[->,>=latex,dashed] (S) to (TS) ;
	\node (S2) at (-2,-5) {$\mathcal{S}^2$};
	\node (T2) at (2,-5) {$\mathcal{T}^2$};
	\draw[->,>=latex] (S) to (S2) ;
	\draw[->,>=latex] (S2) to (TS) ;
	\draw[->,>=latex,dashed] (S2) to (T2) ;
	\draw[->,>=latex] (TS) to (T2) ;
	\draw[->,>=latex,dashed] (T) to (T2) ;
	\node (ST2) at (1.5,-3) {$\mathcal{ST}^2$};
	\node (T3) at (2.75,-2.85) {$\mathcal{T}^3$};
	\draw[->,>=latex,dashed] (TS) to (ST2) ;
	\draw[->,>=latex,dashed] (T2) to (T3) ;
	\draw[->,>=latex] (T2) to (ST2) ;
	\draw[<->,>=latex] (ST2) to (T3) ;
	\draw[->,>=latex,dashed] (ST2) to [out=50,in=95,looseness=8] (ST2) ;
	\draw[->,>=latex,dashed] (T3) to [out=40,in=90,looseness=7] (T3) ;
\end{tikzpicture}
\end{center}
\caption{The Cayley graph of the monoid generated by the relations $\mathcal{S}$ and $\mathcal{T}$. }
\label{fig:STCayley}
\end{figure}
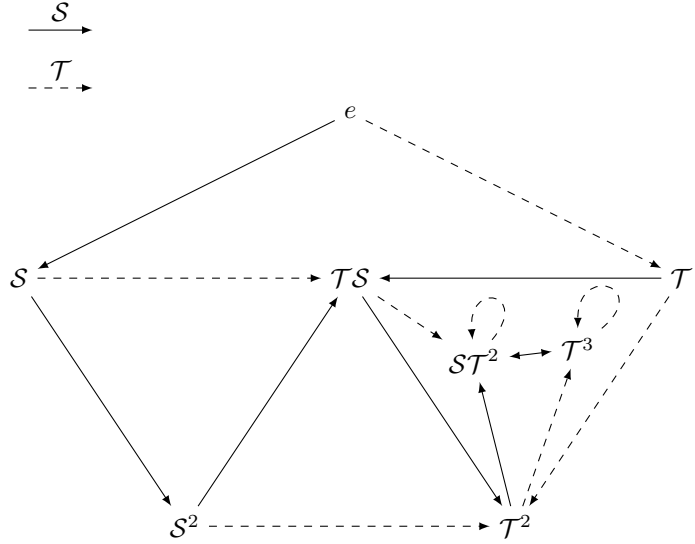

Transformational analysis using relational PK-Nets can be performed in the context of the monoid $M_\mathcal{T}$, but it should be noticed that given two chords $x$ and $y$ there may not always exist a relation $\mathcal{R}$ in $M_\mathcal{T}$ such that $x \mathcal{R} y$.
Therefore, it may be interesting to combine both the $\mathcal{S}$ and the $\mathcal{T}$ relations in order to describe the relations between chords by a series of one or two pitch class movements by semitones. We are thus interested in the structure of the monoid $M_\mathcal{ST}$ generated by the relations $\mathcal{S}$ and $\mathcal{T}$. 

\begin{proposition}
The monoid $M_{\mathcal{ST}}$ generated by the relations $\mathcal{S}$, and $\mathcal{T}$ has for presentation
$$M_{\mathcal{ST}} = \langle \mathcal{S}, \mathcal{T} \mid \mathcal{TS}=\mathcal{ST}, \hspace{0.2cm} \mathcal{S}^3=\mathcal{ST}, \hspace{0.2cm} \mathcal{T}^4=\mathcal{T}^3, \hspace{0.2cm} \mathcal{TS}^2=\mathcal{T}^2, \hspace{0.2cm} \mathcal{ST}^3=\mathcal{ST}^2 \rangle,$$
and contains eight elements.
\end{proposition}

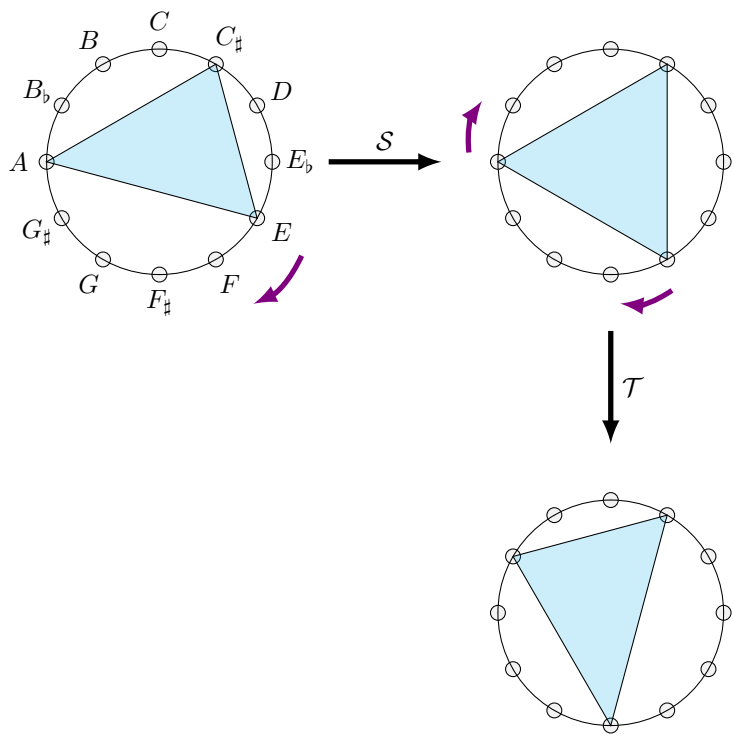
\begin{figure}
\begin{center}
\begin{tikzpicture}[scale=0.75]
	\draw (-5,0) circle[radius=2.0]; 
	\foreach \i in {0,...,11} {
		\node[circle,draw=black,fill=black, fill opacity = 0.05, inner sep=1pt, minimum size=3pt] (A\i) at ({-5.0+2.0*sin(\i*30)},{2.0*cos(\i*30)}) {.};
	}
	\foreach \i in {0,...,11} {
		\node (B\i) at ({-5.0+3.0*sin(\i*30)},{3.0*cos(\i*30)}) {};
	}
	\foreach \i / \name in {0/$C$,1/$C_\sharp$,2/$D$,3/$E_\flat$,4/$E$,5/$F$,6/$F_\sharp$,7/$G$,8/$G_\sharp$,9/$A$,10/$B_\flat$,11/$B$} {
		\node (N\i) at ({-5.0+2.5*sin(\i*30)},{2.5*cos(\i*30)}) {\name};
	}
	\draw[ draw=black, fill=cyan, fill opacity=0.2] (A9.center) -- (A1.center) -- (A4.center) --cycle;
	\draw[ -latex, line width=2, color=violet] (B4) to[bend left=20] (B5); 
	
	\draw[->,>=latex, color=black,line width=2,] (-2.0,0.0) to node[above,midway]{$\mathcal{S}$} (0,0) ;
	
	\draw (3,0) circle[radius=2.0]; 
	\foreach \i in {0,...,11} {
		\node[circle,draw=black,fill=black, fill opacity = 0.05, inner sep=1pt, minimum size=3pt] (C\i) at ({3+2.0*sin(\i*30)},{2.0*cos(\i*30)}) {.};
	}
	\foreach \i in {0,...,11} {
		\node (D\i) at ({3+2.5*sin(\i*30)},{2.5*cos(\i*30)}) {};
	}
	\draw[ draw=black, fill=cyan, fill opacity=0.2] (C9.center) -- (C1.center) -- (C5.center) --cycle;
	\draw[ -latex, line width=2, color=violet] (D5) to[bend left=20] (D6); 
	\draw[ -latex, line width=2, color=violet] (D9) to[bend left=20] (D10); 
	
	\draw[->,>=latex, color=black,line width=2,] (3.0,-3.0) to node[right,midway]{$\mathcal{T}$} (3.0,-5.0) ;
	
	\draw (3.0,-8) circle[radius=2.0]; 
	\foreach \i in {0,...,11} {
		\node[circle,draw=black,fill=black, fill opacity = 0.05, inner sep=1pt, minimum size=3pt] (E\i) at ({3.0+2.0*sin(\i*30)},{-8+2.0*cos(\i*30)}) {.};
	}
	\foreach \i in {0,...,11} {
		\node (F\i) at ({3.0+2.5*sin(\i*30)},{-8+2.5*cos(\i*30)}) {};
	}
	\draw[ draw=black, fill=cyan, fill opacity=0.2] (E6.center) -- (E10.center) -- (E1.center) --cycle;
	
\end{tikzpicture}
\end{center}
\caption{Transformational analysis of the progression $A$ major to $F_\sharp$ major in the context of the $M_{\mathcal{ST}}$ monoid. The semitone changes are indicated by the violet arrows.}
\label{fig:ST-Progression}
\end{figure}

The Cayley graph of this monoid is represented on Figure \ref{fig:STCayley}. As a quick application of this monoid, consider the above mentionned example for the relation between $A_M$ and ${F_\sharp}_M$ and let $\mathbf{C}$ be the new monoid $M_\mathcal{ST}$. An enumeration of the elements of this monoid yields that only $F(f)=\mathcal{TS}$ and $F(f)=\mathcal{ST}^2$ yield valid choices for $F$. A possible path for the progression from $A_M$ to ${F_\sharp}_M$ is shown in Figure \ref{fig:ST-Progression}, which shows the successive relations $A_M \mathcal{S} F_\text{aug}$ and $F_\text{aug} \mathcal{T} {F_\sharp}_m$.

\section{Conclusions}

We have presented in this work a new framework, called relational PK-Nets, in which we consider diagrams in $\mathbf{Rel}$ rather than $\mathbf{Sets}$ as an extension of our previous work on PK-Nets. We have shown how relational PK-Nets capture both the group-theoretical approach and the relational approach of transformational music theory. In particular, we have revisited the parsimonious relations $\mathcal{P}_{m,n}$ of Douthett by studying the structure of monoids based on the $\mathcal{P}_{1,0}$ or $\mathcal{P}_{2,0}$ relations (or subrelations of these), and their corresponding functors to $\mathbf{Rel}$ relating major, minor, and augmented triads. Further perspectives of relational PK-Nets include their integration for computational music theory, providing a way for the systematic analysis of music scores.

\section*{Annex}

It is known that the action of a group, or more generally of a category $\mathbf{C}$, on a set is in 1-1 correspondence with a discrete fibration over $\mathbf{C}$. When replacing functors from $\mathbf{C}$ to $\mathbf{Sets}$ by lax functors to $\mathbf{Rel}$ we have the following proposition.

\begin{proposition}
 There is a 1-1 correspondence $h$ from the set of lax functors $\mathbf{C} \to \mathbf{Rel}$ to the set of faithful functors to $\mathbf{C}$.
\end{proposition}
\begin{proof}
\begin{itemize}
\setlength\itemsep{0.05em}
\item{
Let $S \colon \mathbf{C} \to \mathbf{Rel}$ be a lax functor. We define a faithful functor $h(S) \colon \mathbf{H(S)}$ to $\mathbf{C}$ as follows.
The category $\mathbf{H(S)}$ is the category which has the elements $s$ of the sets $S(e)$ as objects, with $e$ being an object of $\mathbf{C}$, and triplets $(s,g,s')$ as morphisms, with $g \colon e \to e'$ being a morphism of $\mathbf{C}$, $s$ being an element of $S(e)$, $s'$ being an element of $S(e')$, and such that we have $s S(g) s'$.
The composition of two morphisms $(s,g,s')$ and $(u,g',s'')$ is defined as $(u,g',s'')(s,g,s') = (s,g'g,s'')$ iff $s'=u$.
The resulting morphism is well-defined since $s S(g) s'$ and $s' S(g') s''$ imply $s S(g'g) s''$ by definition of the composition of relations.
The associated functor $h(S)$ to $S$ is the functor from $\mathbf{H(S)}$ to $\mathbf{C}$ which sends an object $s$ of $\mathbf{H(S)}$ to the associated object $e$ of $\mathbf{C}$, and sends a morphism $(s,g,s')$ of $\mathbf{H(S)}$ to the corresponding morphism $g$ of $\mathbf{C}$.
This associated functor $h(S) \colon \mathbf{H(S)} \to \mathbf{C}$ is faithful.
}
\item{
Conversely, a faithful functor $k \colon \mathbf{K} \to \mathbf{C}$ is the image by $h$ of the lax functor $S’$ such that $S’(e)$ consists of the objects $s$ of $K$ mapped on $e$ by $k$, and the relation $S'(g)$ for $g \colon e \to e'$ is defined by $s S'(g) s'$ iff there is an $f \colon s \to s'$ in $K$ with $k(f) = g$.
}
\end{itemize}
\end{proof}

Let $\text{LDiag}(\mathbf{Rel})$  be the category of lax diagrams to $\mathbf{Rel}$ which has for objects the lax functors to $\mathbf{Rel}$ and for morphisms from $S \colon \mathbf{C} \to \mathbf{Rel}$ to $S' \colon \mathbf{C'} \to \mathbf{Rel}$ the couples
$(L,\lambda)$ where $L \colon \mathbf{C} \to \mathbf{C'}$ is a functor and $\lambda \colon S \to S'L$ is a natural transformation (i.e., is a lax natural transformation such that $\lambda_e: S(e) \to S'L(e)$ is a map for each object $e$ of $\mathbf{C}$). Let $\mathbf{Faith}$ be the category having for objects the faithful functors between small categories and for morphisms the commutative squares of functors between them.

\begin{proposition}
The 1-1 correspondence $h$ extends into an isomorphism from the category $\text{LDiag}(\mathbf{Rel})$ to the category $\mathbf{Faith}$.
\end{proposition}
\begin{proof}
Let $(L,\lambda) \colon S \to S'$ be a morphism of $\text{LDiag}(\mathbf{Rel})$. We define a functor $h(L, \lambda) \colon \mathbf{H(S)} \to \mathbf{H(S')}$ which maps $(s, g, s')$ on $(\lambda_e(s), L(g), \lambda_{e'}(s'))$.
Thus $(L, h(L, \lambda))$ defines a commutative square of functors from $h(S)$ to $h(S')$.
\end{proof}


\label{lastpage}


\begin{thebibliography}{}

\bibitem[1]{Lewin1982}
Lewin, D.: Transformational Techniques in Atonal and Other Music Theories. Perspectives of New Music, 21(1-2), pp. 312-371 (1982).

\bibitem[2]{Lewin1987}
Lewin, D.: Generalized Music Intervals and Transformations. Yale University Press. New Haven (1987).

\bibitem[3]{Nolan2007}
Nolan, C.: Thoughts on Klumpenhouwer Networks and Mathematical Models: The Synergy of Sets and Graphs. Music Theory Online, 13(3), (2007).

\bibitem[4]{Lewin1990}
Lewin, D.: Klumpenhouwer Networks and Some Isographies That Involve Them. Music Theory Spectrum, 12(1), pp. 83-120 (1990).

\bibitem[5]{Klumpenhouwer1991}
Klumpenhouwer, H.: A Generalized Model of Voice-Leading for Atonal Music. PhD Dissertation, Harvard University (1991).

\bibitem[6]{Klumpenhouwer1998}
Klumpenhouwer, H.: The Inner and Outer Automorphisms of Pitch-Class Inversion and Transposition. Int\'egret, 12, pp. 25-52 (1998).

\bibitem[7]{Mazzola-Andreatta}
Mazzola, G., Andreatta, M.: From a categorical point of view : K-nets as limit denotators. Perspectives of New Music, 44(2), pp. 88-113 (2006).

\bibitem[8]{PopoffMCM2015}
Popoff, A., Andreatta, M., Ehresmann, A.: A Categorical Generalization of Klumpenhouwer Networks, Proceedings of the MCM 2015 Conference, Collins et al. (Eds), LNCS 9110, Springer, pp. 303-314, 2015.

\bibitem[9]{Popoff2016}
Popoff, A., Agon, C., Andreatta, M., Ehresmann, A.: From K-nets to PK-nets: a categorical approach, to appear in Perspectives of New Music.

\bibitem[10]{Noll2005}
Noll, Th.: The Topos of Triads. Colloquium on Mathematical Music Theory, H. Fripertinger and R. Reich (Eds.), Grazer Math. Ber., 347, pp. 103-135 (2005).

\bibitem[11]{Douthett1998}
Douthett, J., Steinbach, P.: Parsimonious Graphs: A Study in Parsimony, Contextual Transformations, and Modes of Limited Transposition. Journal of Music Theory, 42(2), pp. 241-263 (1998).

\bibitem[12]{Cohn2012}
Cohn, R.: Audacious Euphony: Chromaticism and the Triad's Second Nature. Oxford University Press, 2012.

\bibitem[13]{GAP4}
The GAP~Group, \emph{GAP -- Groups, Algorithms, and Programming, Version 4.8.4}; 2016, \verb+(http://www.gap-system.org)+.

\end{thebibliography}
\end{document}